\documentclass[12pt]{amsart}

\numberwithin{equation}{section}

\usepackage{amsfonts,amssymb,amsmath,amsthm}
\usepackage{url}
\usepackage{enumerate}
\usepackage{color}
\usepackage{algorithm}
\usepackage{algorithmic}
\usepackage[colorlinks]{hyperref}

\theoremstyle{plain}
\newtheorem{theorem}{Theorem}[section]
\newtheorem{corollary}[theorem]{Corollary}
\newtheorem{proposition}[theorem]{Proposition}
\newtheorem{lemma}[theorem]{Lemma}

\theoremstyle{definition}
\newtheorem{definition}[theorem]{Definition}

\theoremstyle{remark}
\newtheorem{remark}[theorem]{Remark}
\newtheorem{example}[theorem]{Example}

\begin{document}
\title[]{On almost revlex ideals with Hilbert function of complete intersections}

\author{Cristina Bertone}
\address{Dipartimento di Matematica \lq\lq G. Peano\rq\rq, Universit\`{a} degli Studi di Torino, Italy}
\email{cristina.bertone@unito.it}

\author{Francesca Cioffi%$^{1}$
}
%\thanks{$^{1}$corresponding author}
\address{Dipartimento di Matematica e Applicazioni \lq\lq R. Caccioppoli\rq\rq, Universit\`{a} degli Studi di Napoli Federico II, Italy}
\email{cioffifr@unina.it}

\begin{abstract}
In this paper, we investigate the behavior of almost reverse lexicographic ideals with the Hilbert function of a complete intersection. More precisely, over a field $K$, we give a new constructive proof of the existence of the almost revlex ideal $J\subset K[x_1,\dots,x_n]$,  with the same Hilbert function as a complete intersection defined by $n$ forms of degrees $d_1\leq \dots \leq d_n$. Properties of the reduction numbers for  an almost revlex ideal have an important role in our inductive and constructive proof, which is different from the more general construction given by Pardue in \cite{Pardue}.
We also detect  several cases in which an almost revlex ideal having the same Hilbert function as a complete intersection 
corresponds to a singular point in a Hilbert scheme. This second result is the outcome of a more general study of lower bounds for the dimension of the tangent space to a Hilbert scheme at stable ideals, in terms of the number of minimal generators.
\end{abstract}

\keywords{Almost revlex ideal, reduction number, complete intersection, Hilbert scheme}

\date{\today}
\maketitle

%%%%%%%%%%%%%%%%%%
%% Introduction %%
%%%%%%%%%%%%%%%%%%
%
%Controlli da fare a tappeto o considerazioni generali
%\begin{itemize}
%\item``Sia $H^{[n]}$ la funzione di Hilbert di un'intersezione completa''. Controllare che sia sempre scritto cosi'
%\item controllare dove uso espansione che ci sia scritto il grado giusto
%\item controllare che $\min(\tau) $ sia sempre usato come indice
%\item nell'intro e poi ripetere che seppure la nostra costruzione dell'almost, e anche il conteggio dei generatori minimali, e' \lq\lq incluso\rq\rq nei risultati di Pardue, ci interessa una dimostrazione che usi nell'induzione solo funzioni di Hilbert di intersezioni complete, per in futuro lavorare su Moreno Socias, capendo quindi bene cosa succede agli ideali quando aggiungo l'$n$-esimo generatore in grado $d_n$.
%%\item bisogna dire cosa intendiamo per ``$H$ admits an ideal'' (fatto)
%%\item all'inizio del background definiamo regolarita' di I. Ci serve davvero? sì, lasciamola
%%\item negli articoli sui numeri di riduzione, definiscono anche $r_0$? non ne sono sicura, se ``no'' possiamo sbrigarcela definendolo noi
%\end{itemize}
%
%Rosso per parti ultime modificate.

%Proposte:
%\begin{itemize}
%\item[-] titolo: (1) On almost revlex ideals with Hilbert function of complete intersections. (2) On almost revlex ideals for complete intersections: construction and singularity in a Hilbert scheme.
%\end{itemize}

\section*{Introduction}

In this paper, we investigate the behavior of almost reverse lexicographic ideals with the Hilbert function of a complete intersection. %\begin{color}{blue}We obtain a new construction of these ideals and detect  several cases in which they correspond to singular points in a Hilbert scheme. This second result is the outcome of a more general study of lower bounds for the dimension of the tangent space to a Hilbert scheme at stable ideals, in terms of the number of minimal generators.\end{color}

Referring to \cite{Kunz}, recall that a proper ideal $I$ in a Noetherian ring is called a {\em complete intersection} if the length of the shortest system of minimal generators of $I$ is equal to the height of $I$. A proper ideal $I$ that is generated by a regular sequence in a Noetherian ring is a complete intersection and the converse holds if the ring is Cohen-Macaulay, like a polynomial ring over a field.  Moreover every ideal in a Noetherian ring has a system of generators containing a complete intersection with the same dimension.

The existence of the almost reverse lexicographic ideal with a given Hilbert function is interesting in the study of general schemes with a given Hilbert function in Algebraic Geometry. 
For example, Moreno-Soc\'{i}as' conjecture (see \cite[Conjecture 4.1]{MS}) states that the generic initial ideal (with respect to degrevlex term order) of a polynomial ideal $I\subset R:=K[x_1,\dots,x_n]$ generated by $r$ generic forms is the almost reverse lexicographic ideal $J$ such that  the Hilbert function  of $R/J$ is the same  as that of $R/I$%, which is the Hilbert function of a complete intersection if $r$ is not greater than the number of variables
 (see \cite{AAL, MS, HSW, ChoPark, CG, Cimp} for some partial solutions to this conjecture). It is noteworthy that Moreno-Soc\'{i}as' conjecture implies Fr\"oberg's conjecture (see \cite{Froberg,Pardue,ChoPark} and the references therein, and \cite{N,Trung} for some very recent contributions on the latter conjecture).
Artinian almost reverse lexicographic ideals %$J$
 in $n$ variables also have a significant role in the study of the Lefschetz properties, because they have the maximal possible Betti numbers among all Artinian polynomial ideals that satisfy these properties with order $n$ \cite{Co,HW2009}.

Our investigation focuses on almost reverse lexicographic ideals (Definition \ref{def:revlex}) with the Hilbert function of a complete intersection and  starts from the following more general observation. 
Let $J$ be an Artinian strongly stable ideal in the polynomial ring $R$. It is straightforward that there exists an integer $\ell$ such that, for every $t\geq \ell$, every term of degree $t$ outside $J$ is divisible by the smallest variable. As a consequence, the Hilbert function of $R/J$ is decreasing from $\ell$ on. The minimal integer $\ell$ with this property is strictly connected  with the first reduction number of an Artinian $K$-algebra.

In case we deal with the Hilbert function $H$ of a complete intersection defined by $n$ forms of degrees $d_1 \leq \dots \leq d_n$ in $R$, we describe an explicit construction of the almost reverse lexicographic ideal $J\subseteq R$ such that $H$ is the Hilbert function of $R/J$ (Theorem~\ref{th:main}). The minimal integer $\ell$ with the above property has an important role  in the proof of this result. 
More precisely, properties of reduction numbers of almost reverse lexicographic  ideals are crucial in the inductive and constructive proof we provide, together with the combinatorial properties of the first expansion of the sous-escalier of  a stable ideal  (see \cite{MR2}) and the particular structure of the Hilbert function of a complete intersection (see~\cite{RRR,Pardue}). Partial results of our construction have been presented in \cite{BC-EACA}.

In \cite[Theorems 4 and 5, Corollary 6]{Pardue}, K. Pardue gave a complete characterization of the Hilbert functions that admit almost reverse lexicographic ideals, and among them there are the Hilbert functions of complete intersections. Our proof follows a different path from that  used by Pardue  thus still providing a new insight into the case of complete intersections. 

From our study of the reduction numbers for an almost revlex ideal  $J$, a closed formula for the number of minimal   generators of $J$  arises (see Theorem~\ref{thm:mingen}). 
% This formula also follows from the construction of almost revlex ideals that has given by K.~Pardue in \cite[Theorem 4]{Pardue} and, in the Artinian case, from those of \cite{Co,HW2009} (see also \cite[Section 4]{HW2009})
Using this formula, and a lower bound on the dimension of the Zariski tangent space to a Hilbert scheme at a point corresponding to a stable ideal (Corollary \ref{cor:lower bound}), we exhibit several cases in which  the point corresponding to an Artinian almost reverse lexicographic ideal with the Hilbert function of a complete intersection is singular in the Hilbert scheme (see Section %s \ref{sec:application} and
  \ref{sec:singrevlex}). The main tools for this result are taken from the more general study of marked schemes over quasi-stable ideals \cite{BCR_Macaulay}, similarly to \cite[Section 6]{CLMR2011} for reverse lexicographic ideals with the Hilbert function of general points.

%%%%%%%%%%%%%%%%
%% Background %%
%%%%%%%%%%%%%%%%

\section{Background}

% Pardue: infinite field
% RRR: any field
% Un campo infinito serve solo se lavoriamo con forme generali. Altrimenti possiamo prendere un campo qualsiasi. Decidiamo di lasciare "campo infinito" perche' usiamo il Theorem 5 di Pardue, nonostante sembri che questo teorema valga su un campo qualsiasi: CHIEDERE a Pardue?
Let $R:=K[x_1,\dots,x_n]$ be the polynomial ring over a field $K$ in $n$ variables, endowed with the degree reverse lexicographic  (degrevlex, for short)  term order $\succ$ with $x_1\succ \dots\succ x_n$. In this setting, for every term $\tau:=x_1^{\alpha_1}\dots x_n^{\alpha_n}\not= 1$ we let $\deg(\tau):=\sum_{i=1}^n \alpha_i$ be its degree and we define $\min(\tau):=\min\{x_i \ \vert \ \alpha_i\not=0 \}$, which is the minimal variable that appears in $\tau$ with a non-null exponent. If $\min(\tau)=x_i$, we say that \lq\lq $\tau$ has minimal variable $x_i$\rq\rq.  Similarly, we define  $\max(\tau):=\max\{x_i \ \vert \ \alpha_i\not=0 \}$. Let $\mathbb T$ be the set of the terms of $R$. 

For every integer $t$, we denote by $R_t$ the $K$-vector space of homogeneous polynomials of $R$ of degree $t$, for every subset $\Gamma\subset R$, we let $\Gamma_t:=\Gamma\cap R_t$. %\begin{color}{red}and denote by $\langle \Gamma_t\rangle$ the $K$-vector space generated by $\Gamma_t$. NON SI USA MAI\end{color}

If $I\subset R$ is a homogeneous ideal of $R$, we denote by $I_{\leq t}$ the ideal generated in $R$ by the homogeneous polynomials of $I$ of degree $\leq t$. Moreover, we denote by $\mathrm{reg}(I)$ the {\em regularity of $I$}, that is the minimal integer $m$ such that the $h$-th syzygy module of $I$ is generated in degrees $\leq m+h$, for every $h\geq 0$.

We refer to \cite{Stanley78,Valla} for definitions and results about Hilbert functions of standard graded $K$-algebras. When the $K$-algebra is $R/I$ for a homogeneous ideal $I\subseteq R$, we denote by $H_{R/I}$ its Hilbert function.
We define $\Delta^0 H_{R/I}(t):=H_{R/I}(t)$ and, for every $1\leq i\leq n$,  $\Delta^i H_{R/I}(0):=1$ and  $\Delta^i H_{R/I}(t):= \Delta^{i-1}H_{R/I}(t)-\Delta^{i-1}H_{R/I}(t-1)$ for $t>0$. We call $\Delta^iH(t)$ the {\em $i$-th derivative of $H$}. 

Given a monomial ideal $J\subset R$, we denote by $B_J$ the minimal monomial basis of $J$ and by $\mathcal N(J)$ the {\em sous-escalier} of $J$, that is the set  of terms of $R$ outside $J$. Recall that, for every integer $t$, the cardinality of $\mathcal N(J)_t$ coincides with the value of the Hilbert function of $R/J$ at the degree $t$. 

\begin{definition}
A monomial ideal $J\subseteq R$ is
\begin{description}
\item[{quasi-stable}] if for every $\tau \in J$ and for every $x_j\succ\min(\tau)$, there is some $s>0$ such that the term $x_j^s\tau/\min(\tau)$ belongs to $J$;
 \item[{stable}] if for every $x^\alpha \in J$ and $x_j\succ \min(\tau)$, the term $x_j\tau/\min(\tau)$ belongs to $J$;
 \item[{strongly stable}] if, for every term $\tau\in J$, variable $x_i$  by which $\tau$ is divisible  and variable $x_j\succ x_i$, the term ${x_j \tau}/{x_i}$ belongs to $J$.
\end{description}
\end{definition}

In the hypothesis that the ideal $J$ is stable, we have the following result, that will be used in Section \ref{sec:application}.
\begin{lemma}\label{lemma:varie}
Let $J\subset R$ be a stable ideal. Then 
$$\vert \mathcal N(J)\cap (J : x_n) \vert= \vert \{\tau \in B_J : \tau/x_n \in \mathbb T\}\vert.$$
\end{lemma}

\begin{proof}
It is enough to observe that $x^\beta \in \mathcal N(J)\cap (J:x_n)$ if and only if $x_nx^\beta \in B_J$ (for example, see \cite[Lemma 3(ii)]{Ber}).
\end{proof}

Let $J\subset R$ be any monomial ideal and $t$ an integer. Then, the {\em first expansion} of $\mathcal N(J)_t$ is $\mathcal E(\mathcal N(J)_t):=\mathbb T_{t+1}\setminus (\{x_1,\dots,x_n\}\cdot J_t)=\mathcal N(J_{\leq t})_{t+1}$ (see for instance \cite{MR2}).

If $J$ is a  stable  ideal,  for every integer $t$ the first expansion of $\mathcal N(J)_t$  can be directly computed without repetitions and in increasing order with respect to the reverse lexicographic order as follows, where in square brackets we denote a list of terms of $\mathbb T_t$ that is increasingly ordered with respect to $\succ$:
\begin{equation}\label{eq:espansione}
\mathcal E(\mathcal N(J)_{t})=\bigsqcup_{i=0}^{n-1} \ x_{n-i}\cdot [\tau \in \mathcal N(J)_t :   \min(\tau) \succeq x_{n-i}].
 \end{equation}
Thus, if $\ell$ is an integer such that $\mathcal N(J)_\ell\cap K[x_1,\dots,x_{n-1}]=\emptyset$ and $H$ is the Hilbert function of $R/J$, then for every $t\geq \ell$ we immediately obtain:
\begin{enumerate}[(i)]
\item $\mathcal N(J)_t\cap K[x_1,\dots,x_{n-1}]=\emptyset$,
\item $H(t)\geq H(t+1)$. 
\end{enumerate}

%dimostrazione
% se $\tau\in \bigsqcup_{i=0}^{n-1} \ x_{n-i}\cdot [\tau \in \mathcal N(J)_t :   \min(\tau) \succeq x_{n-i}]$ sia $\tau'\in \mathcal N(J)_t$, $x_j\leq \min(\tau)$ tali che $\tau=x_j\tau'$. Supponiamo che $\tau \notin \mathcal E(\mathcal N(J)_t)$: esiste $\sigma \in J_t$ che divide $\tau$. Essendo $J$ stabile, in particolare esiste $\sigma \in J_t$ e $x_i\leq \min(\sigma)$ tali che $x_j\tau'=x_i\sigma$. Ma allora le variabili minime di questi due termini sono uguali, e otterrei $x_j=x_i$ e quindi $\tau'=\sigma$ assurdo!
%Se $\tau \in \mathcal E(\mathcal N(J)_t)$, allora sia $x_j=\min(\tau)$ e $\tau'=\tau/x_j\in \mathcal N(J)_t$. Scrivendo $\tau=x_j \tau'$ vedo che $\tau$ appartiene all'insieme a secondo membro della formula

\begin{definition}\label{def:segment}
A subset $L\subset \mathbb T_t$ is a {\em reverse lexicographic segment} ({\em revlex segment}, for short) if, for every $\tau\in L$ and $\tau' \in \mathbb T_t$, $\tau'\succ \tau$ implies that $\tau'$ belongs to $L$.

\noindent A monomial ideal $J\subset R$ is a {\em reverse lexicographic ideal} ({\em revlex ideal, for short}) if $J_t\cap \mathbb T$ is a revlex segment, for every degree $t$. 
\end{definition}

\begin{remark}
There are several types of so-called {\em segments} (see  for instance \cite{CLMR2011}). Among them, a reverse lexicographic segment has a special place. 
%,\begin{color}{red} togliere? when it exists \end{color}(see \cite{Deery}).
For example, the generic initial ideal of general points with respect to the degree reverse lexicographic term order is a revlex ideal \cite{MR}. 
\end{remark}

\begin{definition} \label{def:revlex}\cite{Deery}
A monomial ideal $J\subset R$ is an {\em almost reverse lexicographic ideal} (or {\em weakly reverse lexicographic ideal} or {\em almost revlex ideal}, for short) if, for every minimal generator $\tau\in B_J$ of $J$  and  $\tau'\in \mathbb T_{\deg(\tau)}$, $\tau' \succ \tau$ implies that $\tau'$ belongs to $J$.
\end{definition}

%If $J, J'\subset R$ are almost revlex ideals with $H_{R/J}=H_{R/J'}$, then $J=J'$ by the definition. 

\begin{example}\label{ex:almrev}
A revlex ideal is almost revlex, but an almost revlex ideal is not  in general revlex. % For example,  
%let $H$ be the Hilbert function of the general rational curve of degree $5$ in $\mathbb P^3_K$. For the  almost revlex ideal $J=(x_2^3,x_2^2x_1,x_2x_1^2,x_1^3,x_3^2x_1^2)\subset K[x_1,\dots,x_4]$, the Hilbert function of $K[x_1,\dots,x_4]/J$ is $H$. 
The ideal $J=(x_2^3,x_2^2x_1,x_2x_1^2,x_1^3,x_3^2x_1^2)\subset K[x_1,\dots,x_4]$ is not a revlex ideal because $J_{4}\cap \mathbb T$ is not a revlex segment.  
\end{example}

\begin{remark} \label{rem:varie almost}  \cite[Definitions 8 and 10, and Remark 11]{HW2009}
\begin{enumerate}[(i)]
\item \label{it:varAlm_i} An almost revlex ideal is strongly stable. %, in particular a revlex ideal is strongly stable.
\item\label{it:varAlm_ii} If $J$ is an almost revlex ideal in the polynomial ring $K[x_1,\dots,x_{n-1}]\subset R$, then the ideal $JR$ is an almost revlex ideal too.
\end{enumerate}
\end{remark}

We say that a Hilbert function $H$ \emph{admits  an  almost revlex ideal} if there exists the almost revlex ideal $J$ such that $H$ is the Hilbert function of $R/J$. 
Indeed, %if an Hilbert function admits two almost revlex ideals, then they are equal  \cite[Remark 11]{HW2009}. 
if $J, J'\subset R$ are almost revlex ideals with the same Hilbert function, in the sense that $H_{R/J}(t)=H_{R/J'}(t)$ for all $t$, then $J=J'$ \cite[Remark 11]{HW2009}.

%%%%%%%%%%%%%%%%%%%%%%%%%%%%%%%%%%%%%%%%%%%%%%%%%%%%%%%%%%%%%%%%%%
%% Preliminaries on Hilbert functions of complete intersections %%
%%%%%%%%%%%%%%%%%%%%%%%%%%%%%%%%%%%%%%%%%%%%%%%%%%%%%%%%%%%%%%%%%%

\section{Preliminaries on Hilbert functions of complete intersections}

The Hilbert function of a complete intersection is well known, as well as the minimal free resolution, which is a Koszul complex.
In this section, we collect some known properties of the Hilbert function of a complete intersection that have important consequences for our aims. 

We start with a very classical result that connects the Hilbert function of a $K$-algebra with the Hilbert function of a general hypersurface section. %(for example, see \cite[Lecture 13]{Harris}). 

\begin{lemma}\label{lemma:solita seq}
If $I\subset R$ is a homogeneous ideal and $F\in R$ is a form of degree $d$ that is not a zero-divisor in $R/I$, then $H_{R/I}(t)-H_{R/I}(t-d)=H_{R/(I,F)}(t)$, for every $t$. 
\end{lemma}

\begin{proof} It is enough to apply the additive property of a Hilbert function on the short exact sequence \ 
$0 \longrightarrow \left({R}/{I}\right)_{t-d} {\buildrel {\cdot F} \over \longrightarrow } \ \left({R}/{I}\right)_{t} \longrightarrow \left({R}/{(I+(F))}\right)_{t} \longrightarrow 0$. 
\end{proof}

Let $d_1 \leq \dots \leq d_{n-1}\leq d_n$ be $n$ positive integers. We assume $d_1\geq 2$   because, if $d_1=1$, we can rephrase  the framework we are interested in using one less variable and forms of degrees $d_2\leq \cdots \leq d_n$.% potremmo ridurre a una variabile meno e un generatore in meno

For every $1\leq i \leq n$, let $H^{[i]}$ be the Hilbert function of a complete intersection generated by a regular sequence of $i$ forms of degrees $d_1\leq \dots\leq d_i$ in $K[x_1,\dots,x_i]$.
Moreover, for every $1\leq i\leq n$, we let $m_i:=(\sum_{j=1}^{i} d_j) -i$. % and, for $i>1$, $\bar u_i:=\min\left\{\lfloor \frac{m_i}{2}\rfloor,m_{i-1}\right\}$, being $\bar u_1:=0$.
Note that $m_i+1$ is the regularity of the ideal generated by a regular sequence of polynomials of degrees $d_1\leq \cdots\leq d_i$ in $K[x_1,\dots,x_i]$. 

\begin{theorem}\label{th:symmetry} \cite[Theorem (Hilbert Functions under Liaison)]{DGO}
The Hilbert function $H^{[i]}$ is symmetric and $\max\{t \ \vert \ H^{[i]}(t)\not=0\} =m_i$. 
\end{theorem} 

\begin{proposition}\label{prop:inductive step}
For every $2\leq i\leq n$%$i\in \{2,\dots n\}$
, we have
\begin{equation*}\label{eq:relazione funzioni}
H^{[i]}(t)=\sum_{j=0}^t H^{[i-1]}(j)-\sum_{j=0}^{t-d_i} H^{[i-1]}(j).
\end{equation*}
In particular, $\Delta H^{[i]}(t) = H^{[i-1]}(t)-H^{[i-1]}(t-d_i)$.
\end{proposition}

\begin{proof}
It is enough to apply Lemma \ref{lemma:solita seq} to regular sequences.
\end{proof}

The following result is a weaker version of \cite[Theorem 1 and Corollary 2]{RRR}, where the behavior of the Hilbert function of a complete intersection is precisely described. Here, we only recall the properties we need.  
We define $\bar u_1:=0$ and, for every $i>1$, $\bar u_i:=\min\left\{\lfloor \frac{m_i}{2}\rfloor,m_{i-1}\right\}$.

\begin{theorem}\label{th:functions} \cite[Theorem 1 and Corollary 2]{RRR}
The Hilbert function $H^{[i]}$ is strictly increasing in the range $[0,\bar u_i]$ and is decreasing in the range $[\bar u_{i},m_{i}]$. 
\end{theorem} 
%
%\begin{proof} We can refer to .\begin{color}{red}TOGLIEREI  in toto la dimostrazione, non c'\'e scritto nulla di pi\'u di ci\'o che \'e gi\'a nel cappello\end{color}
%\end{proof}

For a Hilbert function $H$, let $\delta$ be the Krull dimension of a graded $K$-algebra having Hilbert function $H$. Then, for every $s\geq \delta$, we let
$$c_s(H):=\max\{ c \ \vert \ \Delta^s H(j) >0, \ \forall \ 0\leq j \leq c\} \text{ (see \cite[Theorem 5]{Pardue}).}$$
We write $c_s$ if it is clear from the context which Hilbert function is involved. If $H=H^{[i]}$, we let $c^{[i]}_s := c_s(H^{[i]})$.

In Theorem \ref{th:functions} a decreasing behavior of the Hilbert function of a complete intersection is described. The following relevant result, which is due to Pardue, highlights that also the derivatives of such a function have a decreasing behavior.

\begin{theorem}\label{th:decrescita} \cite[Theorem 5]{Pardue}
 If $0\leq t \leq c^{[n]}_s$ and $\Delta^{s+1}H^{[n]}(t)\leq 0$, then we also have $\Delta^{s+1}H^{[n]}(t+1)\leq 0$.
\end{theorem}

\begin{example}\label{ex:primo}
Let $n=4$. For $d_1=4$, $d_2=5$, $d_3=7$, $d_4=8$, we obtain $m_1=3$, $m_2=7$, $m_3=13$, $m_4=20$, $\bar u_2=3$, $\bar u_3=6$, $\bar u_4=10$ and 

\begin{small}
%%\begin{table}
%%\caption{}
%%\label{example}
\begin{tabular}{l| rrrrrrrrrrrrrrrr }
$t$   & $0$ & $1$ & $2$ & $3$ & $4$ & $5$ & $6$ & $7$ & $8$ & $9$ & $10$ & $11$ & $12$ & $13$ & $14$ & $15$ \\
\hline
$H^{[4]}(t)$  & $1$ &$4$ &$10$ &$20$ &$34$ &$51$ &$70$ &$89$ &$105$ &$116$ &$120$ &$116$ &$105$ &$89$ &$70$  & \dots\\
$H^{[3]}(t)$ & $1$ & $3$ & $6$ & $10$ & $14$ & $17$ & $19$ & $19$ & $17$ & $14$ & $10$ & $6$ & $3$ & $1$ & $0$ & \dots \\
$\Sigma_{j=0}^t H^{[3]}(j)$ & $1$ & $4$ & $10$ & $20$ & $34$ & $51$ & $70$ & $89$ & $106$ & $120$ & $130$ & $136$ & $139$ & $140$ & $140$ &\dots \\
$\Delta H^{[4]}(t)$ & $1$ & $3$ & $6$ & $10$ & $14$ & $17$ & $19$ & $19$ & $16$ & $11$ & $4$ & $-4$ & $-9$ & $-16$ & $-11$ & \dots \\
\end{tabular}
%%\end{table}
\end{small}\\
Observe that $H^{[4]}(t)=\sum_{j=0}^t H^{[3]}(t)$ for every $t<d_4$. Furthermore in this case $c_0^{[4]}= m_4=20$, $c_1^{[4]}=10$, $c_2^{[4]}=6$.
\end{example}

%%%%%%%%%%%%%%%%%%%%%%%%%%%%%%%%%%%%%%%%%%%%%%%%%
%% Reduction numbers of an almost revelx ideal %%
%%%%%%%%%%%%%%%%%%%%%%%%%%%%%%%%%%%%%%%%%%%%%%%%%

\section{Almost revlex ideals: reduction numbers and minimal generators}\label{sec:almostrevlex}

Using \cite[Corollary 1.4]{HoaTrung}, we can consider the following definition for reduction numbers  for  strongly stable ideals (about  the more general definition and properties of reduction numbers see \cite{HoaTrung} and the references therein).

\begin{definition}
Let $J\subset R=K[x_1,\dots,x_n]$ be a strongly stable ideal  and $\delta$ the Krull dimension of $R/J$. For any $s\geq \delta
%\dim_{Krull}(R/J)
$, we denote by $r_{s}(R/J)$ the {\em $s$-reduction number} of $R/J$, that is  $\min\{t \ \vert \ x_{n-s}^{t+1}\in J \}$. If from the context it is clear what strongly stable ideal is involved, we write $r_s$ only. 
\end{definition}

\begin{remark}\label{rem:zerodivisori}
Let $J\subset R$ be a strongly stable ideal and $\delta$ the Krull dimension of $R/J$.  
\begin{enumerate}[(i)]
\item\label{it:zerodiv_i} $r_s\leq r_{s-1}$, for every $s> \delta$.
\item\label{it:zerodiv_ii} If $J$ is also almost revlex, for every $n-\delta+2 \leq j\leq n$ the variable $x_j$ is not a zero-divisor on $R/J$. 
\end{enumerate}
\end{remark}

In this section, we highlight some results about reduction numbers for an almost revlex ideal  $J$ that can be deduced   from the combinatorial structure of $J$.  Moreover, we deduce a closed formula for the number of minimal generators.

\begin{lemma} \label{lemma:deltaH}
Let $J\subset R$ be an almost revlex ideal, $\delta$ the Krull dimension   and $H$ the Hilbert function of $R/J$. For every $s>\delta$, the Hilbert function of $R/(J + (x_{n-s+1},\dots,x_n))$ coincides with $\Delta^s H(t)$ at every $t\leq r_s$. 
%and  $r_s\leq c_s$.
\end{lemma} 

\begin{proof}
We first consider the Artinian case $\delta=0$. 
Recall that $r_0\geq r_{1}\geq \dots \geq r_{s-1} \geq r_{s}$. We argue by induction on $s$. 

Let $s=1$ and consider $A=R/J$. Suppose that, for some integer $t\leq r_{1}$, the term $x_{n}\tau$ belongs to $J_t$, while $x_{n}$ and $\tau$ do not belong to $J$. Since $J$ is strongly stable, this means that $x_{n}\tau$ is a minimal monomial generator of $J$. Thanks to Definition \ref{def:revlex}, this would imply that $x_{n-1}^t\in J$, because $x_{n-1}^t\succ x_{n}\tau$ for every $\tau\in \mathbb T_{t-1}$, in contradiction with the definition of $r_{1}$.
This means that the variable $x_{n}$ is not a zero-divisor on $A_t$, for every $t\leq r_{1}$. Hence, we can conclude by Lemma \ref{lemma:solita seq} because the short sequence \ 
$0 \longrightarrow A_{t-1} {\buildrel {\cdot x_{n}} \over \longrightarrow } \ A_{t} \longrightarrow \left({A}/{(x_n)}\right)_{t} \longrightarrow 0$ \ is exact for every $t\leq r_{1}$.

For every $s>1$, we apply the same argument to $A={R}/{(J+(x_{n-s+2},\dots,x_n))}$ observing that the variable $x_{n-s+1}$ is not a zero-divisor on $A_j$, for every $j\leq r_s\leq r_{s-1}$. 
%The fact that $r_s\leq c_s$ is now % last affirmation is a consequence of the definition of $c_s$.

Let us now pass to the case $\delta>0$. If $\delta=1$ then consider $A=R/J$, if $\delta>1$ then consider $A=R/(J+(x_{n-\delta+2},\dots,x_{n}))$. In both cases $A$ has Hilbert function $\Delta^{\delta-1}H$ thanks to Remark \ref{rem:zerodivisori}\eqref{it:zerodiv_ii} and Lemma \ref{lemma:solita seq}. Thus, we can proceed like in the Artinian case.
\end{proof}

\begin{lemma}\label{lemma:ri<ci}
Let $J\subset R$ be an almost revlex ideal, $\delta$ the Krull dimension and $H$ the Hilbert function of $R/J$. If $r_s<r_{s-1}$ for some $s> \delta$, then $\Delta^{s}H(t) \leq 0$ for every $r_{s} < t \leq r_{s-1}$. 
%and $r_s=c_s$. 
\end{lemma}

\begin{proof}
If $\delta=0$ and $s=1$, consider $A={K[x_1,\dots,x_n]}/{J}$ and observe that, for every $t>r_1$, $\left({K[x_1,\dots,x_n]}/{(J+(x_n))}\right)_{t}$ vanishes. So, we have the short exact sequence
$$0 \longrightarrow \left({K[x_1,\dots,x_n]}/{J:(x_n)}\right)_{t-1} \longrightarrow A_{t-1} {\buildrel {\cdot x_n} \over \longrightarrow } \ A_{t}  \longrightarrow 0$$ 
and obtain the thesis for every $r_1 < t \leq r_0$.
If $\delta=0$ and $s>1$, we can apply the same argument to $A={K[x_1,\dots,x_n]}/{(J+(x_{n-s+2},\dots,x_n))}$, which has Hilbert function $\Delta^{s-1}H(t)$ for every $t\leq r_{s-1}$ by Lemma \ref{lemma:deltaH}, because for every $t>r_s$ the quotient  $\left({K[x_1,\dots,x_n]}/{(J+(x_{n-s+1},\dots,x_n))}\right)_t$ vanishes.
If $\delta>0$ we argue in the same way considering: $A=R/J$, if $\delta=1$ and $s=2$; $A=K[x_1,\dots,x_n]/(J+x_{n-\delta+2},\dots,x_n)$, which has Hilbert function $\Delta^{\delta-1}H$ thanks to Remark \ref{rem:zerodivisori}\eqref{it:zerodiv_ii} and Lemma \ref{lemma:solita seq}, if $\delta>1$ and $s=\delta+1$; $A=K[x_1,\dots,x_n]/(J+(x_{n-s+2},\dots,x_n))$, otherwise.
\end{proof}

\begin{remark}
At a first glance, the result of Lemma \ref{lemma:ri<ci} could seem similar to that of Theorem \ref{th:decrescita}, but the context and the aim are different. Statement of Theorem \ref{th:decrescita} (and its proof in \cite{Pardue}) considers the Hilbert function of a complete intersection and describes the decreasing behavior of its derivatives, without assuming that this function admits an almost revlex ideal. Statement of Lemma \ref{lemma:ri<ci} considers an almost revlex ideal  and relates the reduction numbers of the almost revlex ideals to the integers $c_s(H)$.
\end{remark}

\begin{proposition}\label{prop:r=c}
Let $J\subset R$ be an almost revlex ideal, $\delta$ the Krull dimension and $H$ the Hilbert function of $R/J$. For every $s\geq \delta$, we have $r_s =c_s$.
\end{proposition}

\begin{proof}
 First, we observe that:
\begin{enumerate}[(a)]
\item\label{it:pra)} the result of Lemma \ref{lemma:deltaH} implies $r_s\leq c_s$ for every $s>\delta$;
\item\label{it:prb)} item\eqref{it:pra)} and the result of Lemma \ref{lemma:ri<ci} imply $r_s=c_s$ if $r_s<r_{s-1}$ for some $s>\delta$.
\end{enumerate} 
We proceed by induction on $s$. For the base of induction we distinguish several cases.

If $s=\delta=0$ then $r_0=\mathrm{reg}(J)-1=c_0$.

If $s=\delta=1$ then we consider $A=R/J$ and $\Delta H(r_1+1)=H(r_1+1)-H(r_1)$. Recall that $H(t)$ counts the number of terms in $\mathcal N(J)_t$. The $H(r_1+1)$ terms in $\mathcal N(J)_{r_1+1}$ are all divisible by $x_n$ because $x_{n-1}^{r_1+1}$ belongs to $J$. Moreover, from \eqref{eq:espansione} we have that all terms in $\mathcal N(J)_{r_1+1}$ are obtained multiplying terms in $\mathcal N(J)_{r_1}$ by $x_n$. Thus, $H(r_1+1)=\vert\mathcal N(J)_{r_1+1}\vert \leq \vert \mathcal N(J)_{r_1}\vert = H(r_1)$.  Hence, $c_1<r_1+1$ and by item \eqref{it:pra)} we have $r_1=c_1$.

If $s=\delta>1$ then we consider $A=R/(J+(x_{n-\delta+2},\dots,x_{n}))$ which has Hilbert function $\Delta^{\delta-1}H$ thanks to Remark \ref{rem:zerodivisori}\eqref{it:zerodiv_ii} and Lemma \ref{lemma:solita seq}, like in the proof of Lemma \ref{lemma:deltaH}. Hence, we can proceed as in the case $s=\delta=1$ replacing the variable $x_n$ by $x_{n-\delta+1}$.

Assume now $s>\delta$. If $r_s<r_{s-1}$, then we can apply item \eqref{it:prb)}. If $r_s=r_{s-1}$, we have $r_s=r_{s-1}=c_{s-1}$ by the inductive hypothesis and, hence, $\Delta^{s-1}H(r_s)>0$ and $\Delta^{s-1}H(r_s+1)\leq 0$, so that $\Delta^s H(r_s+1)=\Delta^{s-1}H(r_s+1)- \Delta^{s-1}H(r_s)<0$. We can now conclude that $r_s=c_s$.
\end{proof}

\begin{remark}\label{rem:c1 complete intersection}
For every strongly stable ideal $J$ with the Hilbert function of an Artinian complete intersection,  we immediately obtain $c_1=\bar u_n$ from Theorem \ref{th:functions}. If $J$ is also almost revlex ideal, then $r_1=c_1=\bar u_n$ from Proposition \ref{prop:r=c}.
\end{remark}

\begin{example}\label{ex:quintica}
Let $I$ be the defining ideal of the general space rational curve of degree~$5$.
In $R=K[x_1,x_2,x_3]$, for the strongly stable ideals $J=(x_1^3,x_1^2x_2,x_1x_2^2,x_2^3,x_1^2x_3^2)$ and $J'=(x_1^3,x_1^2x_2,x_1x_2^2,x_1^2x_3,x_3x_2^3,x_2^4)$ the $K$-algebras $R/J$ and $R/J'$ have the same Hilbert function $H$  as $K[x_1,x_2,x_3,x_4]/(I,x_4)$:

\begin{small}
%\begin{table}
%\caption{}
%\label{example}
\begin{tabular}{l| rrrrrrc }
$t$   & $0$ & $1$ & $2$ & $3$ & $4$ & $5$ &$6$\\
\hline
$H(t)$  & $1$ &$3$ &$6$ &$6$ &$5$ &$5$ & $\dots$\\
\end{tabular}
%\end{table}
\end{small}
\vskip 2mm
\noindent 
 
We have $c_1(H)=c_2(H)=2$. 
For the ideal $J$, which is almost revlex (see Example \ref{ex:almrev}), we have $r_1=r_2=2=c_1(H)=c_2(H)$   and for  the ideal $J'$ we have  $r'_1=3$ and $r'_2=2$. 
\end{example}

\begin{remark} \label{rem:funzione di Pardue}
Let $J\subset R$ be an almost revlex ideal. 
For every integer $s> \delta$, consider the ideal $\bar J = {(J+(x_{n-s+1},\dots,x_n))}/{(x_{n-s+1},\dots,x_n)}$ in $K[x_1,\dots,x_{n-s}]$. From Proposition \ref{prop:r=c}, the Hilbert function of $K[x_1,\dots,x_{n-s}]/\bar J$ is the function $\vert \Delta^s H (t)\vert$ that is defined in \cite[Section~3]{Pardue} in the following way: 
$\vert \Delta^s H (t)\vert =\Delta^s H(t)$, if $\Delta^s H(j)>0$ for $0\leq j\leq t$, and  $\vert \Delta^s H (t)\vert=0$ otherwise.
\end{remark}

From the above considerations, an exact closed formula for the number of the minimal generators of an almost revlex ideal follows in terms of the Hilbert function only. 

\begin{theorem}\label{thm:mingen}
Let $J\subset R$ be an almost revlex ideal and $B_J$ its minimal monomial basis, $\delta$ the Krull dimension and $H$ the Hilbert function of $R/J$. Then, 
\begin{equation}\label{eq:numbermingen}
\vert B_J \vert =\left\{\begin{array}{lr} \sum_{s=0}^{n-1} \Delta^{s} H(c_{s+1}), &\text{ if } \delta=0\\ 
\sum_{s=\delta}^{n-1} \Delta^{s} H(c_{s+1}) + \Delta^{\delta-1}H(c_\delta)-\Delta^{\delta-1}H(\varrho), &\text{ if } \delta>0
\end{array}\right. 
\end{equation}
where $\varrho=\min\{t : \Delta^{\delta-1}H(j)=\Delta^{\delta-1}H(j+1), \forall j\geq t\}$.
\end{theorem}

\begin{proof}
We start detecting the minimal generators with minimal variables $x_1,\dots,x_{n-\delta}$, respectively.
For every $\delta \leq s\leq n-1$, the minimal generators with minimal variable $x_{n-s}$ have degree between $r_{s+1}+1$ and $r_s+1$.

Consider first the case $s=\delta$. If $\delta=0$ then let $A:=R/J$, and if $\delta>0$ then let $A:={R}/{(J+(x_{n-\delta+1},\dots,x_n))}$ which has Hilbert function $\Delta^\delta H(t)$, for every $t\leq r_{\delta}$, because the variables $x_{n-\delta+1},\dots,x_n$ form a regular sequence for $A_{\leq r_\delta}$. For every $t\geq r_{\delta+1}+1$, we have the short exact sequence
$$0 \rightarrow \left({A}/{(0 :_A (x_{n-\delta}))}\right)_{t-1} \rightarrow A_{t-1} {\buildrel {\cdot x_{n-\delta}} \over \longrightarrow } A_t \rightarrow 0$$
because $\left({A}/{(x_{n-\delta})}\right)_t=0$, for every $t\geq r_{\delta+1}+1$. 
Then, we find that the minimal generators of degree $t$ with minimal variable $x_{n-\delta}$ are $\Delta^\delta H(t-1)-\Delta^\delta H(t)=-\Delta^{\delta+1}H(t)$ for every $r_{\delta+1}+1\leq t \leq r_\delta$ (also see Lemma \ref{lemma:ri<ci}) and $\Delta^\delta H(r_\delta)$ at degree $r_\delta+1$, because %all terms in $A_{r_\delta}$ belong to the kernel of the multiplication by $x_{n-\delta}$, by construction, that is 
the short exact sequence becomes  
$$0 \rightarrow \left({A}/{(0 :_A (x_{n-\delta}))}\right)_{r_\delta} \rightarrow A_{r_\delta} \rightarrow 0.$$
So, the number of minimal generators of $J$ with minimal variable $x_{n-\delta}$ is $\Delta^\delta H(r_{\delta+1})$ because
\begin{multline*}-\Delta^{\delta+1}H(r_{\delta+1}+1)-%\Delta^{\delta+1}H(r_{\delta+1}+2)-
 \dots - \Delta^{\delta+1}H(r_\delta)+\Delta^\delta H(r_\delta)=
\\=\Delta^{\delta}H(r_{\delta+1})-\Delta^{\delta}H(r_{\delta+1}+1)%+\Delta^{\delta}H(r_{\delta+1}+1) -\Delta^{\delta}H(r_{\delta+1}+2)
+\dots +\Delta^{\delta}H(r_\delta-1)-\Delta^{\delta}H(r_\delta)+\Delta^\delta H(r_\delta)=\\
= \Delta^\delta H(r_{\delta+1}).
\end{multline*}
We can repeat the above argument for every $\delta< s\leq n-1$, applying Lemma \ref{lemma:deltaH}, so that the minimal generators with minimal variable $x_{n-s}$ are $\Delta^s H(r_{s+1})$. We can then conclude thanks to Proposition \ref{prop:r=c}.
For the case $\delta=0$, this gives the statement on $\vert B_J\vert$. 

For what concerns the case $\delta>0$, we have to carefully consider the fact that $B_J$ in general contains terms with minimal variable $x_{n-\delta+1}$.
%We first consider the case $\delta>0$. 
If $\delta > 1$, by definition of almost revlex ideal the sequence of variables $x_{n-\delta+2},\dots,x_n$ is a regular sequence for $R/J$. Then, we let $A:={R}/{(J+(x_{n-\delta+2},\dots,x_n))}$, which has Hilbert function $\Delta^{\delta-1} H$, thanks to Lemma \ref{lemma:solita seq}. If $\delta=1$, we let $A:=R/J$. The  $(\delta-1)$-th derivative $\Delta^{\delta-1}H$ has constant Hilbert polynomial $p(z)=d$. Thus, if $\varrho$ is the minimal value such that $\Delta^{\delta-1}H(t)=d$, for every $t\geq \varrho$, then $\Delta^{\delta}H(t)=0$ for every $t\geq \varrho+1$.

In this case, the ideal $J$ can have minimal generators with minimal variable $x_{n-\delta+1}$ in the degrees $t\geq r_{\delta}+1=c_\delta+1$. Moreover, for every $t\geq r_{\delta}+1=c_\delta+1$, we have $({A}/{(x_{n-\delta+1})})_t=0$, because $x_{n-\delta}^{r_{\delta}+1}\in J$. Hence, like in the proof of Lemma \ref{lemma:ri<ci}, for every $t\geq r_\delta +1$ we can consider the short exact sequence
$$0 \longrightarrow \left({A}/{(0 :_A (x_{n-\delta+1}))}\right)_{t-1} \longrightarrow A_{t-1} {\buildrel {\cdot x_{n-\delta+1}} \over \longrightarrow } \ A_{t}  \longrightarrow 0.$$
Thus, for every $t\geq r_{\delta}+1=c_\delta+1$, the possible minimal generators of degree $t$ with minimal variable $x_{n-\delta+1}$ are $\Delta^{\delta-1}H(t-1)-\Delta^{\delta-1}H(t)= -\Delta^{\delta} H(t)$.
In conclusion, the number of possible minimal generators of $J$ with minimal variable $x_{n-\delta+1}$ is $0$ if $r_{\delta}+1=c_\delta+1>\varrho$, and is $\sum_{j=c_\delta+1}^\varrho -\Delta^{\delta}H(j)=\Delta^{\delta-1}H(c_\delta)-\Delta^{\delta-1}H(\varrho)$ if $r_{\delta}+1=\varrho$.
\end{proof}

\begin{remark}
We highlight that the case $\delta=0$ of formula \eqref{eq:numbermingen}, as well as the summation from $\delta$ to $n-1$ of the other case, can be also deduced %follows 
from the construction of almost revlex ideals that was given by K.~Pardue in \cite[Theorem 4]{Pardue}. For the Artinian case in characteristic $0$, see also  \cite{HW2009}. As a consequence of Theorem \ref{thm:mingen}, formulas for the Betti numbers of $J$ can be also obtained, because $J$ is strongly stable (for example, see \cite[Section 4]{HW2009}). 
\end{remark}
%This formula gives the exact number of polynomials of the reduced Gr\"obner basis with respect to any graded term order of a homogeneous polynomial ideal with an almost revlex ideal as initial ideal .

%Stesse considerazioni dell'intro. forse sposterei i vari riferimenti in un remark DOPO il corollario (a cui farei fare carriera, vista la dimostrazione potrebbe diventare ``Theorem'')

\begin{example}
Going back to Example \ref{ex:quintica}, we now apply Theorem \ref{thm:mingen} to the ideal $J=(x_1^3,x_1^2x_2,x_1x_2^2,x_2^3,x_1^2x_3^2)\subseteq K[x_1,x_2,x_3]$. In this case we have $\delta=1$, $\varrho=4$, $c_1=c_2=2$ and $c_3=0$. Thus, $\vert B_J\vert = \Delta H(2)+\Delta^2 H(0) + H(2) - H(4)= 3+1+6-5$.
\end{example}

%%%%%%%%%%%%%%%%%%%%%%
%% New construction %%
%%%%%%%%%%%%%%%%%%%%%%

\section{Construction of the almost revlex ideal for $H^{[n]}$}\label{sec:costruzionenuova}

Let $2\leq d_1\leq \cdots \leq d_n$ be integers. In this section, we describe our construction of the almost revlex ideal that is admitted by the Hilbert function $H^{[n]}$. Then, we make a comparison with the construction given in \cite[Theorem 4]{Pardue}.

\begin{theorem}\label{th:main}
For every $n$, the Hilbert function $H^{[n]}$ admits an almost revlex ideal.%, $J^{[n]}$.
\end{theorem}

\begin{proof}
We proceed by induction on $n$. 

For $n=1$, it is sufficient to consider the ideal $(x_1^{d_1})\subseteq K[x_1]$.

%If $n=2$, then it is enough to recall that for every $d_1\leq d_2$, the Hilbert function $H^{[2]}$ admits the almost revlex ideal $J^{[2]}\subset K[x_1,x_2]$  (for example, see \cite[Corollary 2.2]{AAL} and \cite[Proposition 4.2]{MS}):
%$$J^{[2]}=(x_1^{d_1}, x_1^{d_1-1}x_2^{d_2-d_1+1}, x_1^{d_1-2}x_2^{d_2-d_1+3}, \dots, x_1^{d_1-k}x_2^{d_2-d_1+2k-1}, \dots, x_1x_2^{d_1+d_2-3}, x_2^{d_1+d_2-1}).$$

For every $n>1$, by inductive hypothesis the Hilbert function $H^{[n-1]}$ admits an almost revlex ideal,   that we denote by   $J^{[n-1]}\subseteq K[x_1,\dots,x_{n-1}]$. Consider the almost revlex ideal 
$$J':=\left(J^{[n-1]}\right)_{\leq d_n}\cdot K[x_1,\dots,x_n].$$
 The Hilbert function of $K[x_1,\dots,x_n]/J'$ is $H'(t)=\sum_{j=0}^t H^{[n-1]}(j)$, for every $t\leq d_n$, moreover $H'(t)=H^{[n]}(t)$ for every $t\leq d_n-1$ and $H'(d_n)=H^{[n]}(d_n)+1$. We set $c'_i=c_i(H')$ and $r'_i=r_i(R/J')$.

Let $\tau$ be the highest term of degree $d_n$, with respect to the degree reverse lexicographic term order, in $\mathcal N(J')_{d_n}$. Replace the ideal $J'$ by $J'+(\tau)$, so that $J'$ %=J'_{\leq d_n}$
 now is an almost revlex ideal with Hilbert function $H'$ such that $H'(j)=H^{[n]}(j)$ for every $j\leq d_n$. 

For $t>d_n$, assume there is an ideal $J'=J'_{\leq t-1}$ which is an almost revlex ideal with Hilbert function $H'$ such that $H'(j)=H^{[n]}(j)$ for every $j\leq t-1$. 
Let $s$ be  the maximum integer such that there exists a term in $\mathcal N(J')_{t-1}$ that is divisible by the variable $x_{n-s}$. Then, $c^{[n]}_{s+1}=c'_{s+1}=r'_{s+1}<t-1\leq   r'_{s}=c'_{s}\leq c_s^{[n]}$, because $J'$ is almost revlex, hence $\Delta^{s} H^{[n]}(t-1)=\Delta^{s}H'(t-1) > 0$ and $\Delta^{s+1} H^{[n]}(t-1)=\Delta^{s+1}H'(t-1)\leq 0$. 

If $s>0$, from Lemma \ref{lemma:deltaH} and \eqref{eq:espansione} the first expansion of $\mathcal N(J')_{t-1}$ consists of
\vskip 1mm

$H^{[n]}(t-1)$ terms with minimal variable $x_n$

$\Delta H^{[n]}(t-1)$ terms with minimal variable $x_{n-1}$

$\vdots$

$\Delta^s H^{[n]} (t-1)$ terms with minimal variable $x_{n-s}$,

\vskip 1mm
\noindent that are $\sum_{i=0}^s \Delta^i H^{[n]}(t-1)$ terms. Observing that 
\[
H^{[n]}(t)=\Delta^s H^{[n]}(t)+\sum_{j=0}^{s-1}\Delta^j H^{[n]}(t-1),
\]
we obtain 
\[\vert \mathcal E(\mathcal N(J')_{t-1})\vert=H^{[n]}(t)-\Delta^{s+1}H^{[n]}(t).\]
Since $\Delta^{s+1} H^{[n]}(t-1)\leq 0$, thanks to Theorem \ref{th:decrescita} we have  $\Delta^{s+1}H^{[n]}(t)\leq 0$, from which it follows that the cardinality of the first expansion of $\mathcal N(J')_{t-1}$ is higher than or equal to $H^{[n]}(t)$.
%$$\vert \mathcal E(\mathcal N(J')_{t-1})\vert=H^{[n]}(t)-\Delta^{s+1}H^{[n]}(t).$$ 
In this case, let $\tau_1,\dots,\tau_h$ be the highest $h=\sum_{i=0}^s \Delta^i H^{[n]}(t-1) - H^{[n]}(t)=-\Delta^{s+1}H^{[n]}(t)\geq 0$ terms of degree $t$, with respect to the degree reverse lexicographic term order, in the first expansion of $\mathcal N(J')_{t-1}$. Replace the ideal $J'$ by $J'+(\tau_1,\dots,\tau_h)$, so that $J'$ becomes an almost revlex ideal with Hilbert function $H'$ such that $H'(j)=H^{[n]}(j)$ for every $j\leq t$.

We can repeat this construction until we find that the sous-escalier of $J'$ at degree $t-1$ only consists of $H^{[n]}(t-1)$ terms with minimal variable $x_n$, that is $s=0$. In this case, the first expansion of $\mathcal N(J')_{t-1}$ also consists of $H^{[n]}(t-1)$ terms of degree $t$ with minimal variable $x_n$. So, we have $c^{[n]}_1=c'_{1}=r'_{1}<t-1\leq r'_{0}=c'_{0}\leq c_0^{[n]}$. By Theorem \ref{th:functions}, we know that $c^{[n]}_1=\bar u_n$, because $H^{[n]}$ is strictly decreasing from  $\bar u_n$ on, and we can continue this construction up to degree $t=d_1+\dots +d_{n-1}+d_n-n$, obtaining an almost revlex ideal $J'$  having Hilbert function $H^{[n]}$.%, hence $J^{[n]}=J'$. 
\end{proof}

From now, for every $1\leq i \leq n$, we denote by $J^{[i]}$ the almost revlex such that the Hilbert function of $K[x_1,\dots,x_i]/J^{[i]}$ is $H^{[i]}$. Moreover, we let $r^{[i]}_s:=r_s(K[x_1,\dots,x_i]/J^{[i]})$.

The explicit construction of $J^{[n]}$ in the proof of Theorem \ref{th:main} can be algorithmically improved observing that, for every $i \in\{2,\dots,n-1\}$, it is sufficient to compute the generators of $J^{[i]}$ up to degree $d_{i+1}$. 

Summing up, we obtain Algorithm \ref{alg:almost}, for which we assume that the following procedures are available:
\begin{itemize}
\item \textsc{HFunction(\textnormal{$[d_1,\dots,d_j], h$})} takes in input an increasingly ordered list of $j$ positive integers $[d_1,\dots,d_j]$ and an integer $h>0 $, and returns the list of values of the Hilbert function $H^{[j]}(t)$, for every $t\leq h$.  Precisely, for $j=1$, the output is the function that assumes value $1$ for every $t< d_1$ and $0$ otherwise, for $j>1$ the output can be computed for instance by Proposition \ref{prop:inductive step}.
\item \textsc{Greatest(\textnormal{$\mathcal N,h$})} takes in input a set of terms $\mathcal N$ of degree $t$ and returns the greatest $h$ terms w.r.t.~degrevlex.
\end{itemize}

\begin{algorithm}[H]
\caption{\label{alg:almost} $\textsc{AlmostRevLex}(n,[d_1,\dots,d_n])$}
\begin{algorithmic}[1]
\REQUIRE A positive integer $n$
\REQUIRE an increasingly ordered   list of $n$ positive integers $[d_1,\dots,d_n]$,  with $d_1\geq 2$. 
\ENSURE  The ideal $J^{[n]}$.
\STATE $J'\leftarrow (x_1^{d_1})$;
%\STATE $\begin{color}{blue}H^{[1]}\end{color} \leftarrow \textsc{HFunction}([d_1],d_{2})$;
\IF{ $n>1$ }
\STATE  $J' \leftarrow (J')_{\leq d_2}$;
\STATE $d_{n+1}\leftarrow(\sum_i d_i)-n+1$;
\FOR{$i=2,\dots,n$}
\STATE $H^{[i]}\leftarrow \textsc{HFunction}([d_1,\dots, d_i],d_{i+1})$; 
\STATE $J'\leftarrow J'\cdot K[x_1,\dots,x_i]$;
\STATE $\tau\leftarrow \max_{degrevlex} \mathcal N(J')_{d_i}$;
\STATE $J'\leftarrow J'+(\tau)$;
\FOR{$t=d_{i},\dots,d_{i+1}$}
\STATE $x_{n-s}\leftarrow \min_{degrevlex}\{\min(\tau)\vert \tau \in \mathcal N(J')_{t-1}\}$;
\STATE $h=-\Delta^{s+1}H^{[i]}(t)$;
\STATE $[\tau_1,\dots,\tau_h]\leftarrow \textsc{Greatest}(\mathcal N(J')_t,h)$; \STATE\label{endforprinc} $J'\leftarrow J'+(\tau_1,\dots,\tau_h)$;
\ENDFOR
\ENDFOR
\ENDIF
\RETURN $J'$
\end{algorithmic}
\end{algorithm}

\begin{remark}
In general, the ideal $J'=(J^{[i]})_{\leq d_i+1}$ that we obtain at line \ref{endforprinc} is not Artinian in $K[x_1,\dots, x_i]$, hence the Hilbert function of $K[x_1,\dots,x_i]/J'$ is neither $H^{[i]}$ nor $\Delta^i H^{[n]}$.
\end{remark}

\begin{example}\label{ex:constr}
The Hilbert function $H^{[3]}$ of a complete intersection generated by $3$ forms of degrees $d_1=3$, $d_2=d_3=4$ admits the following almost revlex ideal $J^{[3]}\subset K[x_1,x_2,x_3]$:
$$J^{[3]}=(x_{{1}}^{3},x_{{1}}^{2}x_{{2}}^{2},x_{{1}}x_{{2}}^{3},x_{{2}
}^{5},x_{{2}}^{4}x_{{3}},x_{{1}}^{2}x_{{2}}x_{{3}}^{3},x_{{1}}
x_{{2}}^{2}x_{{3}}^{3},x_{{2}}^{3}x_{{3}}^{3},x_{{1}}^{2}x_{
{3}}^{5},x_{{1}}x_{{2}}x_{{3}}^{5},x_{{2}}^{2}x_{{3}}^{5},x_{{1
}}x_{{3}}^{7},x_{{2}}x_{{3}}^{7},x_{{3}}^{9}
)$$
which is constructed by Algorithm \ref{alg:almost} in the following way. 
\begin{itemize}
\item[n=1:] $J^{[1]}=(x_1^3)$;
\item[n=2:] In order to compute $J^{[2]}$ up to degree $d_3$, we add to $J'=(J^{[1]})_{\leq 4}K[x_1,x_2]$ only the term $x_1^2 x_2^2$, and we do not need to explicitly compute the other generators of $J^{[2]}$. Observe indeed that $K[x_1,x_2]/J'$, with $J'=(x_1^3,x_1^2x_2^2)$, does not have Hilbert function $H^{[2]}$,  because  $H_{R/J'}(t)=H^{[2]}(t)$ only up to $t=d_3$.
\item[n=3:] Let now $J':= (x_1^3,x_1^2x_2^2)\cdot K[x_1,x_2,x_3]$. In order to compute $J^{[3]}$ up to degree $d_1+d_2+d_3-3+1=9$, we add to $J'$ the term $x_1 x_2^3$. We consider the highest term $\tau=x_1x_2^3$ of $\mathcal N(J')_4$ with respect to degrevlex and update $J'$ to the ideal $(J')_{\leq 4}+(x_1x_2^3)\subset K[x_1,x_2,x_3]$. In this way, $H'(4)=H^{[3]}(4)$, while $H'(5)-H^{[3]}(5)=11-9=2$.
The first expansion $\mathcal E(\mathcal N(J')_4)$ contains $H^{[3]}(4)=9$ terms with minimal variable $x_3$ and $\Delta H^{[3]}(4)=1$ term with minimal variable $x_2$:
$$x_{{2}}^{5} ,x_{{2}}^{4}x_{{3}},x_{{1}}^{2}x_{{2}}x_{{3}}^{2}, x_{{1}}x_{{2}}^{2}x_{{3}}^{2}
,x_{{2}
}^{3}x_{{3}}^{2},x_{{1}}^{2}x_{{3}}^{3}, 
x_{{1}}x_{{2}}x_{{3}}^{3},x_{{2}}^{2}x_{{3}}^{3},
x_{{2}}x_{{3}}^{4},x_{{1}}x_{{3}}
^{4}
, x_{{3}}^{5}.$$
\noindent
Then we update $J'$ to $J'+(x_{{2}}^{5} ,x_{{2}}^{4}x_{{3}})$, so that $H'(5)=H^{[3]}(5)=9$, while $H'(6)-H^{[3]}(6)=9-6=3$.
We then consider $\mathcal E(\mathcal N(J')_5)$, whose terms  are all divisible by $x_3$, because $\Delta H^{[3]}(6)<0$.  The greatest $3$ terms of  $\mathcal E(\mathcal N(J')_5)$ are  $x_1^2 x_2x_3^3,x_1x_2^2x_3^3, x_2^3x_3^3$. 
Again, we update $J'$ to $J'+(x_1^2 x_2x_3^3,x_1x_2^2x_3^3, x_2^3x_3^3)$, so that $H'(6)=H^{[3]}(6)=6$, while $H'(7)-H^{[3]}(7)=3$.
Iterating this process, we add to $J'$ the terms $x_3^5x_2^2,x_3^5x_2x_1,x_3^5x_1^2$, so that $H'(7)=H^{[3]}(7)$, and the terms $x_3^7x_2,x_3^7x_1$, so that $H'(8)=H^{[3]}(8)$. Finally, the first expansion $\mathcal E(\mathcal N(J)_8)$ contains only the term $x_3^9$, which we add to $J'$, obtaining $J'=J^{[3]}$.
\end{itemize}
\end{example}

In \cite[Theorem 4]{Pardue}, K. Pardue characterizes all the Hilbert functions that admit an almost revlex ideal. The symmetry of a Hilbert function is not a sufficient condition for admitting an almost revlex ideal, indeed. For example, the symmetric sequence ${\bf h}=(1,13,12,13,1)$ is the $h$-vector of Gorenstein ideals (see \cite{Stanley78,AhnShin}), but this Hilbert function does not admit an almost revlex ideal. We can explain this fact either observing that $\bf h$ does not satisfy the conditions of \cite[Theorem 4]{Pardue} or looking at the structure of the expansion of a stable ideal (see formula \eqref{eq:espansione}). %This expansion \begin{color}{red}suggests\end{color} that the $12$ terms of degree $2$ in the sous-escalier of a possible almost revlex $J\subseteq K[x_1,\cdots,x_{13}]$, such that $K[x_1,\cdots,x_{13}]/J$ has Hilbert function $\bf h$, should be divisible by the last variable $x_{13}$, so that we cannot have $13$ terms of degree $3$ in the sous-escalier.

In \cite[Theorem 4]{Pardue}, K. Pardue considers a Hilbert function $H$, the behavior of which has been recalled in Theorem \ref{th:decrescita}, and constructs an almost revlex ideal $J$ such that $H$ is the Hilbert function of $R/J$. He proceeds by induction performing a hyperplane section by $x_n$, so that the new generators to be added have minimal variable $x_n$. Then, he uses the formula of Eliahou and Kervaire \cite{EK} for the Hilbert series of a stable ideal, which is based on the shape of the minimal generators of the ideal, in order to guarantee the correctness of his construction. 

The construction given by K. Pardue in \cite[proof of Theorem 4, (3) $\rightarrow$ (1)]{Pardue} is more general than ours, since it concerns not only the Hilbert functions of complete intersections. Nevertheless, the Hilbert functions and almost revlex ideals that one has to consider  following Pardue's induction are different  from  those involved in the proof of Theorem \ref{th:main}. Indeed, in order to construct an almost revlex ideal in $K[x_1,\dots,x_n]$ having Hilbert function $H$, Pardue considers the almost revlex ideal in $K[x_1,\dots,x_{n-1}]$ having Hilbert function $|\Delta^1 H(t)|$ (see Remark \ref{rem:funzione di Pardue}). Our construction starts from the ideal $J^{[n-1]}$ which has Hilbert function $H^{[n-1]}$. In general, if $H=H^{[n]}$ is the Hilbert function of the complete intersection   with integers $d_1\leq \dots\leq d_n$, then $H^{[n-s]}(t)\neq \vert \Delta^s H^{[n]}(t)\vert$, hence the corresponding almost revlex ideals are different. Furthermore, following Pardue's construction for $H^{[n]}$, the new generators to be added with minimal variable $x_n$ have in general a degree that is higher than $d_n$.

\begin{example}\label{ex:sharp}
The Hilbert function $H^{[3]}$ of a complete intersection generated by $3$ forms of degrees $d_1=3$, $d_2=d_3=4$ admits the  almost revlex ideal $J^{[3]}\subset K[x_1,x_2,x_3]$, whose construction is given in Example \ref{ex:constr} according to Algorithm \ref{alg:almost}.
%$$J^{[3]}=(x_{{1}}^{3},x_{{1}}^{2}x_{{2}}^{2},x_{{1}}x_{{2}}^{3},x_{{2}
%}^{5},x_{{2}}^{4}x_{{3}},x_{{1}}^{2}x_{{2}}x_{{3}}^{3},x_{{1}}
%x_{{2}}^{2}x_{{3}}^{3},x_{{2}}^{3}x_{{3}}^{3},x_{{1}}^{2}x_{
%{3}}^{5},x_{{1}}x_{{2}}x_{{3}}^{5},x_{{2}}^{2}x_{{3}}^{5},x_{{1
%}}x_{{3}}^{7},x_{{2}}x_{{3}}^{7},x_{{3}}^{9}
%)$$
%which is constructed by Algorithm \ref{alg:almost} in the following way. 
Take into account that $J^{[2]}=(x_1^3,x_2^2x_1^2,x_1x_2^4,x_2^6)\subset K[x_1,x_2]$.  We denote by $H'$ the Hilbert function of $K[x_1,x_2,x_3]/J'$ when $J'=(J^{[2]})_{\leq d_3}\cdot K[x_1,x_2,x_3]$ and by $H''$ that of $K[x_1,x_2]/ (J^{[2]})_{\leq d_3}$. Observe that $H'(t)=\sum_{j=0}^t H^{[2]}(j)$ for every $t\leq d_3$, and $H''(t)=H^{[2]}(t)$ for every $t\leq d_3$.

\begin{small}
%\begin{table}
%\caption{}
%\label{example}
\begin{tabular}{r| rrrrrrrrrrl }
$t$   & $0$ & $1$ & $2$ & $3$ & $4$ & $5$ & $6$ & $7$ & $8$ & $9$ & $10$ \\
\hline
$H^{[3]}(t)$ & $1$ & $3$ & $6$ & $9$ & $10$ & $9$ & $6$ & $3$ & $1$ & $0$ & \dots  \\
$H'(t)$ & $1$ & $3$ & $6$ & $9$ & $11$ & $13$ & $15$ & $17$ & $19$ & $21$ & \dots \\
$H^{[2]}(t)$ & $1$ & $2$ & $3$ & $3$ & $2$ & $1$ & $0$ & $0$ & $0$ & $0$ & \dots \\
$H''(t)$ & $1$ & $2$ & $3$ & $3$ & $2$ & $2$ & $2$ & $2$ & $2$ & $2$ & \dots \\
\end{tabular}
%\end{table}
\end{small}

We highlight that the path one follows  by Pardue's construction is different from  the path of our construction: using Theorem \ref{th:main} for every $i\leq n-1$ one considers the Hilbert function $H^{[i]}(t)$ and constructs $(J^{[i]})_{\leq d_{i+1}}$, while using \cite[proof of Theorem 4, (3) $\rightarrow$ (1)]{Pardue} one considers the Hilbert function $\vert \Delta^i H^{[n]}(t)\vert$.
In particular, following  \cite[proof of Theorem 4, (3) $\rightarrow$ (1)]{Pardue}
in order to construct the almost revlex ideal $J^{[3]}\subset K[x_1,x_2,x_3]$, we consider the almost revlex ideal $I=(x_1^3,x_1^2 x_2^2,x_1x_2^3,x_2^5)$  in $K[x_1,x_2]$ having Hilbert function  $\vert\Delta H^{[3]}(t)\vert$:

\begin{small}
\begin{tabular}{r| rrrrrrl }
$t$   & $0$ & $1$ & $2$ & $3$ & $4$ & $5$ & $6$ \\
\hline
%$H(t)$  & $1$ &$3$ &$6$ &$9$ &$10$ &$9$ &$6$ &$3$ &$1$ & $0$ & $\cdots$\\
$\vert\Delta H^{[3]}(t)\vert$ & $1$ & $2$ & $3$ & $3$ & $1$ & $0$  & $\cdots$
\end{tabular}
\end{small}
\vskip 2mm
\noindent Observe that  $\vert\Delta H^{[3]}(t)\vert$ is not the Hilbert function  of a complete intersection because it is not symmetric. 
%with $d_1=3$, $d_2=4$, and $I\neq J^{[2]}$. 
\end{example}

\begin{remark}\label{rem:mingen ci}
The proof of Theorem \ref{th:main} gives rise to the following formula for the number of minimal generators of $J^{[n]}$, which is of course equivalent to that of Theorem \ref{thm:mingen}:
\begin{equation}\label{ex:remark}
\vert B_{J^{ [n] }} \vert = \sum_{s=0}^{n-1} \sum_{j=c_{s+1}+1}^{c_s} -\Delta^{s+1} H^{[n]}(j+1).
\end{equation}
We obtain \eqref{ex:remark} in the following way. 
If $n=1$ then $c_0=d_1$ and $c_1=0$. Further, $\Delta H^{[1]}(t)=0$ for every $1\leq t\leq c_0$ and $\Delta H^{[1]}(t)=-1$. Hence, \eqref{ex:remark} holds for $n=1$.
If $n>1$ and $J':= (J^{[n-1]})_{\leq d_n}K[x_1,\dots,x_n]$ like in the proof of Theorem \ref{th:main}, then
for every $t<d_n$, we have $\Delta^{s+1} H^{[n]}(t)=\Delta^s H^{[n-1]}(t)$ and  $\Delta^{s+1} H^{[n]}(d_n)=\Delta^s H^{[n-1]}(d_n)-1$ thanks to Proposition \ref{prop:inductive step}; further, for every $s\geq 1$, if $c^{[n]}_s<d_n$, then $c^{[n]}_s=c^{[n-1]}_{s-1}$. 
Moreover, up to degree $d_n$ the ideal $J'$ has the same number of minimal generators of $J^{[n-1]}$, that is 
$$\sum_{s=0}^{n-2} \sum_{j=\min\{c^{[n-1]}_{s+1}+1,d_n-1\}}^{\min\{c^{[n-1]}_s,d_n-1\}} -\Delta^{s+1} H^{[n-1]}(j+1)= 
-1+\sum_{s=0}^{n-2} \sum_{j=\min\{c^{[n]}_{s+2}+1,d_n-1\}}^{\min\{c^{[n]}_{s+1},d_n-1\}} -\Delta^{s+2} H^{[n]}(j+1)=$$
$$= -1+ \sum_{s=1}^{n-1} \sum_{j=\min\{c^{[n]}_{s+1}+1,d_n-1\}}^{\min\{c^{[n]}_{s},d_n-1\}} -\Delta^{s+1} H^{[n]}(j+1).$$
At degree $d_n$ the ideal $J^{[n]}$ has one more generator with respect to $J'$. Then, we conclude adding the minimal generators of degrees $j\geq d_n$ as in the proof of Theorem \ref{th:main}. %So, up to degree $d_n$ the number of minimal generators of $J^{[n]}$ is
%$$= \sum_{s=1}^{n-1} \sum_{j=\min\{c^{[n]}_{s+1}+1,d_n-1\}}^{\min\{c^{[n]}_{s},d_n-1\}} -\Delta^{s+1} H^{[n]}(j+1).$$ 
\end{remark}

%%%%%%%%%%%%%%%%%%%%%%%%%%%%%%%%%%%%%%%%%%%%%%%%%%%%%%%%%%%%%%%%%%%%%%%%%%%%%%%%%%%%%%%%%%%%%%%%%
%%  %%
%%%%%%%%%%%%%%%%%%%%%%%%%%%%%%%%%%%%%%%%%%%%%%%%%%%%%%%%%%%%%%%%%%%%%%%%%%%%%%%%%%%%%%%%%%%%%%%%%

\section{A lower bound for the dimension of the tangent space to a punctual Hilbert scheme at stable ideals}\label{sec:application}

In the present section, we start considering Artinian \emph{quasi-stable} ideals $J$ in $R$.

Every quasi-stable ideal $J$ has a special set $\mathcal P(J)$ of monomial generators, in general non-minimal, that is called the \emph{Pommaret basis} of $J$ and allows a unique decomposition of every term $\tau \in J$ in the following sense: 
\begin{itemize}
\item[($\star$)] for every $\tau\in J$, there is a unique $x^\alpha\in \mathcal P(J)$ so that $\tau=x^\alpha x^\delta$ with  $\max(x^\delta)\preceq \min(x^\alpha)$. 
\end{itemize}
Stable ideals are quasi-stable, and strongly stable ideals are stable. If $J$ is stable, then $\mathcal P(J)=B_J$. 

Let $H$ be the Hilbert function of $R/J$ and $S:=R [x_{n+1}]$. The Hilbert polynomial $p(z)$ of $S/(JS)$ is the constant $D:=\sum_{j\geq 0} H(j)$, because $J$ is Artinian. We can then identify $J$ with the point $\mathrm{Proj}(S/(JS))$ of the Hilbert scheme $\mathrm{Hilb}_{D}^n$, which parameterizes flat families of closed subschemes in $\mathbb P^n_K$  with Hilbert polynomial $D$. Hence, we will say that $J$ is (or corresponds to) a point of $\mathrm{Hilb}_{D}^n$.

Our aim is to give a lower bound for the dimension of the Zariski tangent space $\mathcal T_J$ to $\mathrm{Hilb}_{D}^n$ at the point $J$, using techniques and results that have been developed in \cite{BCR_Macaulay, Gore}. A similar investigation has been given in \cite[Lemma 6.1 and Theorem 6.2]{CLMR2011} under the more restrictive hypotheses that $JS$ is a \emph{hilb-segment ideal} with respect to a suitable term order and the field $K$ has characteristic zero. %un almost revlex con fz di Hilbert di intersezione completa è un gen-segmento per definizione, ma potrebbe non essere hilb segmento, qualunque sia il term order. per d_1=3,d_2=d_3=4 l'almost revlex non è hilb segmento. inoltre noi non chiediamo caratteristica 0.
Although we will use the Jacobian criterion in order to compute $\mathcal T_J$, for the moment we can consider any field $K$ because $J$ is a $K$-valued point ($K$-point, for short) of the Hilbert scheme, i.e.~a closed point with residue field $K$.
%% Vedi Foundations di Vakil pag. 326, ExERCISE 12.1.G 

Referring to \cite{BCR_Macaulay,Gore,LR2011}, first we briefly recall how one can obtain a set of equations defining the Zariski tangent space to $\mathrm{Hilb}^n_{D}$ at $J$, but more generally at any point of a suitable open subset of $\mathrm{Hilb}^n_{D}$.
Also recall that $J$ is an Artinian quasi-stable ideal and $\mathcal P(J)$ denotes its Pommaret basis. 

For every $x^\gamma \in \mathcal P(J)$, we define the \emph{marked polynomial} 
$$f_\gamma=x^\gamma+\sum_{x^\beta \in \mathcal N(J)}C_{\gamma\beta} x^\beta \in K[C][x_1,\dots,x_n],$$ 
where $x^\gamma$ is called the {\em head term} of $f_\gamma$ and $C=\{C_{\alpha\beta}\vert x^\alpha \in \mathcal P(J), x^\beta \in \mathcal N(J)\}$ is the set of all parameters  appearing in the marked polynomials over $\mathcal P(J)$% $C_{\gamma\beta}$
. In this context, the coefficient of a term in the variables $x_1,\dots,x_n$ will be called a {\em $x$-coefficient}.

For every $x^\gamma\in \mathcal P(J)$ and  $x_j \succ\min(x^\gamma)$, we consider the polynomial $x_j f_\gamma$ and, by a suitable Noetherian confluent reduction process that is based on property ($\star$) (see \cite[Definition 4.2]{BCR_Macaulay}), compute the polynomials that are involved in the following equality
\begin{equation}\label{eq:ridMarc}
x_j f_\gamma=\sum p_{\alpha'\delta'}x^{\delta'} f_{\alpha'}+ h_{j\gamma},
\end{equation}
where $x^{\alpha'}$ belongs to $\mathcal P(J)$, $\max(x^{\delta'}) \preceq \min(x^{\alpha'})$, $h_{j\gamma}$ is supported on $\mathcal N(J)$, and $p_{\alpha'\delta'}$ belongs to $K[C]$. All polynomials and terms in \eqref{eq:ridMarc} are uniquely determined (see \cite[Proposition 4.3]{BCR_Macaulay}).

Let $\mathcal U\subset K[C]$ be the set consisting of all the $x$-coefficients appearing in the polynomials $h_{j\alpha}$. The ideal generated by $\mathcal U$ in $K[C]$ defines an open affine subscheme of $\mathrm{Hilb}_{D}^n$, that is called \emph{$J$-marked scheme} and usually denoted by $\mathrm{Mf}(J)$  \cite[Propositions 5.6 and 6.13(ii)]{BCR_Macaulay}. 

Being $\mathrm{Mf}(J)\subset \mathbb A^{\vert C\vert}$ an open subscheme of $\mathrm{Hilb}_{D}^n$, we can explicitely compute the Zariski tangent space to $\mathrm{Hilb}_{D}^n$ at any point belonging to $\mathrm{Mf}(J)$ using the polynomials in the set $\mathcal U$, as explained in \cite[Corollary 1.9 and Remark 1.10]{Gore}. For what concerns the quasi-stable ideal $J$, we can simply take the linear part of the polynomials in $\mathcal U$. 

\begin{remark}\label{rem:efficient way}
For an efficient way to only compute the linear part of the polynomials in $\mathcal U$, one can use \cite[Algorithm 2]{LR2011}, which also applies to marked schemes. 
At http://wpage.unina.it/cioffifr/MaterialeAlmostRevLex an implementation of this algorithm for marked schemes on an Artinian quasi-stable ideal is available.
\end{remark}

In order to finally obtain a lower bound for $\dim \mathcal T_J$, we focus on the parameters $C_{\alpha\beta}$ that never appear in the linear part of the polynomials in $\mathcal U$. It is clear that, if $\bar C$ is a set of such parameters, then $\vert C\vert \ \geq \ \dim_K \mathcal T_J \ \geq \ \vert \bar C\vert$ indeed.

\begin{theorem}\label{th:lower bound}
If $J\subset R$ is an Artinian stable ideal, then 
%\begin{enumerate}[(i)]
%\item\label{it:lowb_i} 
for every $x^\beta\in \mathcal N(J)\cap (J:x_n)$ and for every $x^\alpha \in B_J$, the parameter $C_{\alpha\beta}$ does not appear in the linear part of the polynomials in $\mathcal U$.
%\item\label{it:lowb_ii} $\vert B_J\vert \vert\mathcal N(J)\vert \geq \ \dim \mathcal T_{J} \ \geq %\begin{color}{red} TOGLIERE DA ENUNCIATO? \vert B_{J}\vert \cdot \vert \mathcal N(J)\cap (J : x_n) \vert =\end{color}
% \vert B_{J}\vert \cdot \vert \{\tau \in B_J : \tau/x_n \in \mathbb T\}\vert$.
%\end{enumerate}
\end{theorem}

\begin{proof}

%For item \eqref{it:lowb_i}, 
We first observe that $x_n x^\beta \in J$ implies $x_j x^\beta \in J$ for every $j$, because $J$ is a stable ideal. 
Let $x^\alpha$ belong to $B_J$. Our aim is to prove that for every $x^{\gamma}\in B_J$, for every $x_j \succ\min(x^{\gamma})$, in the writing \eqref{eq:ridMarc} for $x_jf_{\gamma}$, the coefficient $C_{\alpha\beta}$ does not appear linearly in the $x$-coefficients of $h_{j\gamma}$.

We first consider the case $\gamma=\alpha$. Recall that $x_jx^\beta$ belongs to $J$. Hence, following the procedure in \cite[Section 4]{BCR_Macaulay} in order to compute the writing \eqref{eq:ridMarc} for $x_jf_\alpha$ with $x_j\succ \min(x^\alpha)$, we find some $x^{\delta'}f_{\alpha'}$ ($\alpha'\not=\alpha$) such that $x_jx^\beta=x^{\alpha'}x^{\delta'}$. Then, $C_{\alpha\beta}x_jx^\beta$ can be rewritten by $C_{\alpha\beta}x^{\delta'}f_{\alpha'}$ and $C_{\alpha\beta}$ appears in $p_{\alpha'\delta'}\in K[C]$ in the right-hand side of \eqref{eq:ridMarc}, but not in the $x$-coefficients of the polynomials $h_{j\alpha}$, in this case.

We now assume that $\gamma\neq \alpha$ and $f_{\alpha}$ is used in the writing \eqref{eq:ridMarc} for $x_jf_\gamma$, with $x_j\succ \min(x^\gamma)$. We first observe that $x_jx^\gamma$ does not belong to $B_J$, otherwise the term $x^\gamma$ of $B_J$ would divide another minimal monomial generator of $B_J$. Hence, according to \cite[Section 4]{BCR_Macaulay}, if $f_\alpha$ is used for rewriting $x_jx^\gamma$, that is $x_jx^\gamma= x^\delta x^{\alpha}$ with $x^\delta\neq 1$, we conclude as in the previous case that $C_{\alpha\beta}$ does not appear linearly in the $x$-coefficients of $h_{j\gamma}$. If $f_\alpha$ is used for rewriting an other term $\tau$, then $\tau$ must to have a non-constant coefficient in $K[C]$ and $C_{\alpha\beta}$ appears in the {\em non-linear} part either of some $p_{\alpha\delta'}\in K[C]$ in the right-hand side of \eqref{eq:ridMarc} or of the $x$-coefficients of $h_{j\gamma}$.
%For what concerns item \eqref{it:lowb_ii}, the first inequality is a consequence of the construction of $\mathcal T_J$ by means of marked schemes. Indeed, the marked scheme is embedded in an affine space of dimension $\vert C\vert =\vert B_J\vert\vert \mathcal N(J)\vert$. The other inequality is a consequence of item \eqref{it:lowb_i} and of Lemma \ref{lemma:varie}. 
\end{proof}

\begin{remark}
Observe that if $x^\beta\in \mathcal N(J)$ and $x_\ell x^\beta \in J$ for some $\ell<n$, we cannot identify any $x^\alpha \in B_J$ such that $C_{\alpha\beta}$ does not appear in the linear part of the $x$-coefficients of the polynomials $h_{j\gamma}$. Indeed, if $x^\delta f_\alpha$ appears in the right-hand side of \eqref{eq:ridMarc} for some $x^\gamma \in B_J$, then $x^\delta<_{lex}x_j$, hence there is no way to guarantee that $C_{\alpha\beta}$ does not appear linearly in the $x$-coefficients of $h_{j\gamma}$.
\end{remark}

\begin{corollary}\label{cor:lower bound}
If $J\subset R$ is an Artinian stable ideal, then 
\[\vert B_J\vert \cdot \vert\mathcal N(J)\vert \geq \ \dim \mathcal T_{J} \ \geq %\begin{color}{red} TOGLIERE DA ENUNCIATO? \vert B_{J}\vert \cdot \vert \mathcal N(J)\cap (J : x_n) \vert =\end{color}
 \vert B_{J}\vert \cdot \vert \{\tau \in B_J : \tau/x_n \in \mathbb T\}\vert.\]
\end{corollary}
\begin{proof}
The first inequality is a consequence of the construction of $\mathcal T_J$ by means of marked schemes. Indeed, the marked scheme is embedded in an affine space of dimension $\vert C\vert =\vert B_J\vert\vert \mathcal N(J)\vert$. The other inequality is a consequence of Theorem \ref{th:lower bound} and of Lemma \ref{lemma:varie}. 
\end{proof}

Thanks to Corollary \ref{cor:lower bound} we now obtain a sufficient condition for $J$ being a singular point in $\mathrm{Hilb}_D^n$ when $J$ is stable and also Borel-fixed over an infinite field $K$. From now, we assume that the field $K$ is infinite, because we will use notions and results that need this hypothesis.

Recall that Borel-fixed ideals are fixed points of an algebraic group action on the Hilbert scheme. More precisely, an ideal $J\subset R$ is Borel-fixed (Borel, for short) if $g(I)=I$ for every element $g$ of the Borel subgroup consisting of the upper triangular matrices on $K$ of order $n$. A Borel ideal is always a monomial quasi-stable ideal (not stable in general), but the property to be Borel depends on the characteristic of the field. In characteristic $0$, Borel ideals and strongly stable ideals coincide. In general, strongly stable ideals are Borel regardless of the characteristic of the field $K$, however there are quasi-stable (resp. stable) ideals that are not Borel, for every characteristic of the field.
%il seguente ideale è Stabile ma non Borel  qualunque sia la caratteristica di K
%\[
%J:=({x_{{1}}}^{3},{x_{{1}}}^{2}{x_{{2}}}^{2},{x_{{1}}}^{2}x_{{2}}x_{{3}},
%{x_{{1}}}^{2}{x_{{3}}}^{3},x_{{1}}{x_{{2}}}^{3},x_{{1}}{x_{{2}}}^{2}x_
%{{3}},x_{{1}}x_{{2}}{x_{{3}}}^{2},x_{{1}}{x_{{3}}}^{4},{x_{{2}}}^{5},{
%x_{{2}}}^{4}x_{{3}},{x_{{2}}}^{3}{x_{{3}}}^{2},{x_{{2}}}^{2}{x_{{3}}}^
%{3},x_{{2}}{x_{{3}}}^{4},{x_{{3}}}^{5}) \subset K[x_1,x_2,x_3]
%\]
%
%il seguente ideale è quasi stabile ma non Borel qualunque sia la caratteristica di K
%J=(x_1^2,x_2)\subset K[x_1,x_2]
%
%
In the study of Hilbert schemes, Borel ideals have a very important role  (see for instance \cite{Hart}).

\begin{theorem}\label{thm:singStable}
Let $K$ be an infinite field. With the above notation, if $J\subset R$ is an Artinian stable Borel-fixed ideal, then 
$$\vert B_{J}\vert \cdot\vert \{\tau \in B_J : \tau/x_n \in \mathbb T\}\vert> n\cdot D \  \Rightarrow \ J \text{ corresponds to a singular point in } \mathrm{Hilb}_{D}^n.$$
\end{theorem}

\begin{proof}
From the proof of \cite[Corollary 19]{PardueThesis} it follows that for every Artinian Borel ideal $J$ there is a component of $\mathrm{Hilb}_{D}^n$ that contains $JS$ and the lex-segment ideal. This component must be the same for all these Borel ideals (see also \cite{ReevesRadius} in characteristic zero), because there exists a unique component containing the lex-segment ideal. Indeed, recall that the lex segment ideal is a smooth point in the Hilbert scheme, hence it lies on a unique component of the Hilbert scheme, which has dimension $n\cdot D$ (see \cite{ReevesStillman}). 
 
Hence if  $\vert B_{J}\vert \cdot \vert \{\tau \in B_J : \tau/x_n \in \mathbb T\}\vert > n\cdot D$, then by Corollary \ref{cor:lower bound} the dimension of the Zariski tangent space to $\mathrm{Hilb}^n_d$ at $J$ is strictly bigger than the dimension of the lex component. So, $J$ is a singular point of this component and of $\mathrm{Hilb}_D^n$.
\end{proof}

\begin{remark}
Observe that the condition that is given in Theorem \ref{thm:singStable} is sufficient only. For instance, the ideal $J^{[3]}$ of Example \ref{ex:sharp} does not satisfy the numerical condition of Theorem~\ref{thm:singStable}, being $\vert B_{J^{[3]}}\vert=14$ and $\vert \{\tau \in B_{J^{[3]}} : \tau/x_3 \in \mathbb T\}\vert=10$. Nevertheless, $J^{[3]}$ corresponds to a singular point of $\mathrm{Hilb}_{48}^3$ because a direct computation gives $\dim \mathcal T_{J^{[3]}}=286>144$, where $144=48\cdot 3$ is the dimension of the lex component of $\mathrm{Hilb}_{48}^3$. We will find some analogous situations in Examples \ref{ex:d=2} and \ref{ex:d=3}.
\end{remark}

\section{Some almost revlex singular points in a Hilbert scheme}
\label{sec:singrevlex}

In this section, over an infinite field $K$ we specialize the results of Section \ref{sec:application} to Artinian almost revlex ideals and find several classes of almost revlex ideals with the Hilbert function of a complete intersection that are singular points in a Hilbert scheme. 
We assume $n\geq 3$, because $\mathrm{Hilb}_D^2$ is irreducible and smooth \cite[Theorem 2.4]{Fo}. 

An almost revlex ideal is strongly stable, hence it is stable and Borel-fixed in every characteristic. %and it lies on the lex component of the Hilbert scheme $\mathrm{Hilb}^n_{D}$, with $D=\vert \mathcal N(J)\vert$.
Using Theorems \ref{thm:singStable} and \ref{thm:mingen}, we obtain a sufficient condition for an Artinian almost revlex ideal $J$ to be a singular point in the Hilbert scheme $\mathrm{Hilb}_D^n$ in terms of the Hilbert function $H$ of $R/J$ only, where recall that $D=\sum_{j\geq 0} H(j)$. We can re-state  Corollary \ref{cor:lower bound} and Theorem \ref{thm:singStable} in the following way, indeed. 

\begin{lemma}\label{lemma:Hc1}
If $H$ is a Hilbert function admitting an Artinian almost revlex ideal $J\subset R$, then 
$\vert\{\tau \in B_J: \tau/x_n \in \mathbb T\}\vert=H(c_1)$.
\end{lemma}

\begin{proof}
The statement follows from the arguments of the proof of Theorem \ref{thm:mingen}.
\end{proof}

\begin{theorem}\label{th:singRevlex}
Let $H$ be a Hilbert function admitting an Artinian almost revlex ideal $J\subset R$ and $D=\sum_{j\geq 0} H(j)$. 
\begin{enumerate}[(i)]
\item\label{th:sing_i} $\dim \mathcal T_J \geq \left(\sum_{s=0}^{n-1} \Delta^s H(c_{s+1})\right) \cdot H(c_1) > H(c_1)^2$.
\item\label{th:sing_ii} if either $\left(\sum_{s=0}^{n-1} \Delta^s H(c_{s+1})\right)\cdot H(c_1)> n\cdot D$ or $H(c_1)^2 \geq n\cdot D$, then $J$ is a singular point in $\mathrm{Hilb}_{D}^n$.
\end{enumerate}
\end{theorem}

\begin{proof}
We can use Theorem \ref{thm:mingen} to write $\vert B_{J}\vert$ in terms of $H$ and its derivatives, because $H$ admits the almost revlex ideal $J$. So, the first inequality of item \eqref{th:sing_i} follows from Corollary \ref{cor:lower bound} and Lemma~\ref{lemma:Hc1}. For the second inequality it is enough to observe that $\vert B_J\vert \geq H(c_1)+n-1$, where $n-1$ counts the minimal generators of $J$ that are powers of the variable $x_1,\dots,x_{n-1}$. 
For item \eqref{th:sing_ii}, thanks to Theorem \ref{thm:singStable} and item \eqref{th:sing_i} we have the thesis. 
\end{proof}

%  eventuale per il futuro: capire delle nostre costruzioni cosa vale per ideali con la proprietà strong lefschetz n volte
%\begin{remark}
%\begin{color}{blue}The statement of Theorem \ref{th:singRevlex} also holds if $J$ is replaced by any monomial ideal satisfying the Strong Lefschetz property with order $n$.\end{color}
%\end{remark}

As in the previous sections, we denote by $H^{[n]}$ the Hilbert function of the Artinian complete intersection generated by polynomials of degrees $2\leq d_1\leq \cdots \leq d_n$ and by $J^{[n]}$ the almost revlex ideal in $R$ such that the Hilbert function of $R/J^{[n]}$ is $H^{[n]}$. In this case, we have $D=\sum_{j\geq 0} H^{[n]}(j)=d_1\cdots d_n$. 

Theorem \ref{th:singRevlex} gives a sufficient condition for $J^{[n]}$ to be a singular point of $\mathrm{Hilb}^n_{D}$, that only involves $H^{[n]}$. 
We now collect some technical results in order to reach our aim.

\begin{lemma}\label{lem:Hc1}
Let $H^{[n]}$ be the Hilbert function of the Artinian complete intersection defined by the positive integers $d_1\leq\dots\leq d_n$. Then
\[H^{[n]}\left(c_1^{[n]}\right) >  \frac{d_1\cdots d_n}{\sum_{i=1}^n d_i}.
\] 
\end{lemma}

\begin{proof}
From Theorem \ref{th:functions}, for every $d_1\leq \dots \leq d_n$, the value  $H^{[n]}\left(c^{[n]}_1\right)$ is the maximum that is assumed by the Hilbert function $H^{[n]}$. Remembering that $(\sum_{i=1}^n d_i)-n$ is the maximum integer at which $H^{[n]}$ assumes a non-null value  and $\sum_{j\geq 0} H^{[n]}(j)=d_1\cdots d_n$, we obtain 
\[H^{[n]}\left(c^{[n]}_1\right)\cdot\left(\sum_{i=1}^n d_i-n+1\right)\geq d_1\cdots d_n.\]
Using the fact $\sum_{i=1}^n d_i-n+1<\sum_{i=1}^n d_i$, we conclude.
\end{proof}
 
\begin{remark}\label{rem:improvement}
We can refine the statement of Lemma \ref{lem:Hc1} in the following way. Observe that $H^{[n]}(j)=\binom{n-1+j}{j}$, for every $0\leq j < d_1$. So, thanks to the symmetry of $H^{[n]}$ we have:
\begin{equation}\label{eq:improvement}
H^{[n]}\left(c_1^{[n]}\right) \geq \frac{d_1\cdots d_n-2\binom{n+d_1-1}{d_1-1}}{\sum_{i=1}^n d_i-n+1-2d_1}.
\end{equation}
\end{remark}

%Questi risultati non sono sorprendenti ma per quanto ne sappiamo non ci sono altre dimostrazioni disponibili in letteratura.
%Possibile Congettura: $J^{[n]}$ \`{e} il punto con maggiore dimensione dello spazio tangente tra gli strongly stable che soddisfano $n$ volte la propriet\`{a} Strong Lefschetz.

\begin{proposition}\label{prop:n-1=n}
Let $d_1\leq\cdots\leq d_n$ be positive integers and $D=d_1\cdots d_n$. The almost revlex ideal $J^{[n]}$ is a singular point of the Hilbert scheme $\mathrm{Hilb}^n_{D}$, if either of the following numerical conditions hold
\begin{enumerate}[(i)]
\item\label{it:n-1=n_i} $d_1\cdots d_n >n\left(\sum_{i=1}^n d_i\right)^2$;
\item\label{it:n-1=n_ii} $d_1\cdots d_{n-1} >n^3 d_n$;
\item \label{it:n-1=n_iii} $d_{n-1}=d_n$ and $d_1\cdots d_{n-2}\geq n^3$.
\end{enumerate}
\end{proposition}

\begin{proof}
By Theorem \ref{th:singRevlex} it is sufficient to prove that $\left(H^{[n]}\left(c_1^{[n]}\right)\right)^2>n\cdot D$ if either of \eqref{it:n-1=n_i}, \eqref{it:n-1=n_ii} or \eqref{it:n-1=n_iii} holds.
If \eqref{it:n-1=n_i} holds, by Lemma \ref{lem:Hc1} we immediately have the thesis.
If \eqref{it:n-1=n_ii} holds, observing that $(\sum_{i=1}^n d_i)^2 \leq n^2 d_n^2$ and using  Lemma \ref{lem:Hc1} we have
\[
\left(H^{[n]}\left(c_1^{[n]}\right)\right)^2>\frac{(d_1\dots d_{n-1} d_n)^2}{n^2 d_n^2} \geq\frac{d_1\cdots d_n n^3 d_n^2}{n^2 d_n^2} =n\cdot D.
\]
Finally, if \eqref{it:n-1=n_iii} holds, we obtain the thesis by the same arguments of the previous case.
\end{proof}

\begin{corollary}\label{cor:corollario 1}
For all integers $ d_1\leq \dots \leq d_{n-1} \leq d_n$, if $d_{n-1}=d_n$ then the almost revlex ideal $J^{[n]}$ corresponds to a singular point in the Hilbert scheme in the following cases
\begin{enumerate}[(i)]
\item\label{it:cor1_i} $d_1\geq 2$ and $n\geq 14$;
\item\label{it:cor1_ii} $d_1\geq 3$ and $n\geq 8$;
\item\label{it:cor1_iii} $d_1\geq 4$ and $n\geq 6$;
\item \label{it:cor1_iv}$d_1\geq 5$ and $n\geq 5$;
\item \label{it:cor1_v}$d_1\geq 8$ and $n\geq 4$;
\item\label{it:cor1_vi} $d_1\geq 27$ and $n\geq 3$.
\end{enumerate}
\end{corollary} 

\begin{proof}
By induction on $n$, we obtain $d_1\cdots d_{n-2}\geq d_1^{n-2}\geq n^3$ in all the cases that are listed in the statement. Then, it is enough to apply Proposition \ref{prop:n-1=n}\eqref{it:n-1=n_iii}. 
\end{proof}

%\begin{remark}\label{rem:generali}
%Ecco alcune considerazioni generali per $d=d_1=d_n$. Recupero alcuni risultati riguardanti l'andamento della funzione di Hilbert di una complete intersection descritto in \cite[Theorem 1]{RRR} che fino ad ora non abbiamo richiamato, considerati nel caso $d_1=d=d_n$. Pongo: $t=\sum_{i=1}^n (d_i-1)=n(d-1)=nd-n$.
%\begin{itemize}
%\item Se $n$ e $(d-1)$ sono dispari, ossia $n$ \`{e} dispari e $d$ \`{e} pari, allora $t=n(d-1)$ \`{e} dispari e $\frac{t+1}{2}=\frac{nd-n+1}{2}\geq d$ per ogni $n$ e per ogni $d$. In questo caso la funzione di Hilbert \`{e} costante solo nell'intervallo $[\frac{t-1}{2},\frac{t+1}{2}]$. 
%\item Se $n$ \`{e} pari oppure $(d-1)$ \`{e} pari, ossia $n$ \`{e} pari oppure $d$ \`{e} dispari, allora $t=n(d-1)$ \`{e} pari e $\frac{t+1}{2}=\frac{nd-n+1}{2}\geq \frac{nd-n}{2}\geq d$ per ogni $n$ e per ogni $d$. In questo caso la funzione di Hilbert non \`{e} mai costante.
%\end{itemize}
%\end{remark}

\begin{corollary}\label{cor:d}
For every $n\geq 3$ and $2\leq d=d_1=d_n$, $J^{[n]}$ corresponds to a singular point in the Hilbert scheme $\mathrm{Hilb}_{d^n}^n$. 
\end{corollary}

\begin{proof}
For $d=2,3,4$ the statement holds thanks to items \eqref{it:cor1_i}, \eqref{it:cor1_ii}, \eqref{it:cor1_iii} of Corollary \ref{cor:corollario 1} and by direct computations that are collected in next Examples \ref{ex:d=2}, \ref{ex:d=3}, \ref{ex:d=4}. For $d\geq 5$ and $n\geq 5$ the statement holds due to item \eqref{it:cor1_iv} of Corollary \ref{cor:corollario 1}. So, the case $d\geq 5$ with $n=3,4$ remains open. 

More precisely, thanks to items  \eqref{it:cor1_v} and \eqref{it:cor1_vi} of Corollary \ref{cor:corollario 1} we have to focus on $d=5,6,7$ with $n=4$ and $5\leq d \leq 26$ with $n=3$. Except for $d=5$ with $n=3$, in each of these cases we obtain the thesis using formula \eqref{eq:improvement} and Theorem \ref{th:singRevlex}. For $d=5$ with $n=3$, we compute $\vert B_{J^{[3]}}\vert \cdot H^{[3]}(c_1^{[3]})=25\cdot 19=475>375=3\cdot 5^3$ and conclude by Theorem \ref{th:singRevlex}
\end{proof}

\begin{example}\label{ex:d=2}
Let $2=d_1=d_n$. We show that $J^{[n]}$ is a singular point for every  $3\leq n\leq 13$. 

If $n$ is even, then $c_1=\frac{n}{2}$ and $H^{[n]}(c_1)=\binom{n}{\frac{n}{2}}$ because

\begin{small}
\begin{tabular}{r| cccccccccc }
$t$   & $0$ & $1$ & $2$ & $\dots$ & $\frac{n}{2}$ & $\frac{n}{2}+1$ & $\dots$&$n-1$&$n$ \\
\hline
$H^{[n]}(t)$ &$1$ & $\binom{n}{1}$ & $\binom{n}{2} $& $\dots$& $\binom{n}{\frac{n}{2}}$ & $\binom{n}{\frac{n}{2}-1}$ & $\dots$ &$\binom{n}{1}$&$ 1$
\end{tabular}
\end{small}

\vskip 1mm
If $n$ is odd, then $c_1=\frac{n-1}{2}$ and $H^{[n]}(c_1)= \binom{n}{\frac{n-1}{2}}$ because

\begin{small}
\begin{tabular}{r| cccccccccc }
$t$   & $0$ & $1$ & $2$ & $\dots$ & $\frac{n-1}{2}$ & $\frac{n+1}{2}$ & $\dots$&$n-1$&$n$ \\
\hline
$H^{[n]}(t)$ &$1$ & $\binom{n}{1}$ & $\binom{n}{2} $& $\dots$& $\binom{n}{\frac{n-1}{2}}$ & $\binom{n}{\frac{n-1}{2}}$ & $\dots$ &$\binom{n}{1}$&$ 1$
\end{tabular}
\end{small}

\noindent If $6\leq n\leq 13$, we obtain $H^{[n]}(c_1^{[n]})^2 > n\cdot 2^n$ by an explicit computation.\\
If $n=5$, we find $\vert B_{J^{[5]}}\vert \cdot H^{[5]}(c_1^{[5]})=18\cdot 10=180 > 160=5\cdot 2^5$.\\ 
If $n=4$, we find $\vert B_{J^{[4]}}\vert \cdot H^{[4]}(c_1^{[4]})=12\cdot 6=72 > 64=4\cdot 2^4$.\\
If $n=3$, we have $\vert B_{J^{[3]}}\vert \cdot H^{[3]}(c_1^{[3]})=6\cdot 3=18<3\cdot 2^3$. Hence, in order to prove that $J^{[3]}$ corresponds to a singular point in the Hilbert scheme $\mathrm{Hilb}_{2^3}^3$, we need to make a direct computation which gives $\dim \mathcal T_{J^{[3]}}=36>24=3\cdot 2^3$.% and conclude by Theorem \ref{thm:singStable}.
\end{example}

\begin{example}\label{ex:d=3}
Let $3=d_1=d_n$. For $3 \leq n\leq 7$ we have the following Hilbert functions, respectively:

\begin{small}
\begin{tabular}{r| cccccccccc }
$t$   & $0$ & $1$ & $2$ & $3$ & $4$ & $5$ & $6$ & $7$ & $8$ \\
\hline
$H^{[3]}(t)$ &$1$ & $3$ & $6 $& $7$& $6$ & $3$ & $1$\\
$H^{[4]}(t)$ &$1$ & $4$ & $10 $& $16$& $19$ & $16$ & $\dots$\\
$H^{[5]}(t)$ &$1$ & $5$ & $15 $& $30$& $45$ & $51$ & $45$ &$\dots$\\
$H^{[6]}(t)$ &$1$ & $6$ & $21 $& $50$& $90$ & $126$ & $141$& $126$& $\dots$\\
$H^{[7]}(t)$ &$1$ & $7$ & $28 $& $77$& $161$ & $266$ & $357$& $393$& $357$ & $\dots$
\end{tabular}
\end{small}

\noindent If $n=3$, then $H^{[3]}(c_1)=7$, $\Delta H^{[3]}(c_2)=3$, $\Delta^2 H^{[3]}(c_3)=1$ and $(7+3+1)\cdot 7=77 < 3\cdot 3^3$, so that in this case we cannot apply Theorem \ref{th:singRevlex}. Nevertheless, we conclude that $J^{[3]}$ corresponds to a singular point of the Hilbert scheme $\mathrm{Hilb}_{27}^3$ because  by a direct computation we obtain $\dim \mathcal T_J= 147 > 3\cdot 27$. \\ 
%Although $J^{[3]}$ is a revlex ideal, this case has not been  considered in \cite{CLMR2011} because  $J^{[3]}$ is not a hilb-segment ideal. \begin{color}{red}non ho capito bene il senso di questa osservazione sul non essere hilb-segment.\end{color}
If $n=4$, then  $H^{[4]}(c_1)^2=19^2 > 324=4\cdot 3^4$.\\
If $n=5$, then $H^{[5]}(c_1)^2=51^2 >1215=5\cdot 3^5$.\\
If $n=6$, then $H^{[6]}(c_1)^2=141^2  > 4374=6\cdot 3^6$.\\
If $n=7$, then $H^{[7]}(c_1)^2=393^2 > 7\cdot 3^7$.

\noindent In conclusion, if $3=d_1=d_n$ and $3\leq n\leq 7$, then $J^{[n]}$ is a singular point thanks to either Theorem \ref{th:singRevlex} or a direct computation of the dimension of the Zariski tangent space.
\end{example}

\begin{example}\label{ex:d=4}
Let $4=d_1=d_n$. For $n=3,4,5$ we have the following Hilbert functions, respectively:
\vskip 1mm
\begin{small}
\begin{tabular}{r| cccccccccccc } 
$t$   & $0$ & $1$ & $2$ & $3$ & $4$ & $5$ & $6$ & $7$ & $8$ & $9$ &$\dots$\\
\hline
$H^{[3]}(t)$ &$1$ & $3$ & $6 $& $10$& $12$ & $12$ & $10$ &$\dots$ & & & \\
$H^{[4]}(t)$ &$1$ & $4$ & $10 $& $20$& $31$ & $40$ & $44$& $40$& $\dots$ & & \\
$H^{[5]}(t)$ &$1$ & $5$ & $15 $& $35$& $65$ & $101$ & $135$& $155$& $155$ & $135$ &$\dots$
\end{tabular}
\end{small}
\vskip 1mm
\noindent If $n=3$, then $H^{[3]}(c_1)=12$, $\Delta H^{[3]}(c_2)=4$, $\Delta^2 H^{[3]}(c_3)=1$ and $(12+4+1) 12 > 192= 3\cdot 4^3$.\\ 
If $n=4$, then $H^{[4]}(c_1)^2=44^2>1024=4\cdot 4^4$.\\
If $n=5$, then $H^{[5]}(c_1)^2=155^2 > 5\cdot 4^5$. 

\noindent In conclusion, if $4=d_1=d_n$ and $n=3,4,5$, $J^{[n]}$ is a singular point due to Theorem \ref{th:singRevlex}.
\end{example}

\section*{Acknowledgment}

The authors are grateful to Margherita Roggero for several useful conversations on the topics of the paper, to Keith Pardue for valuable suggestions on a previous draft, and to Davide Franco for helpful comments.

The authors are members of \lq\lq National Group for Algebraic and Geometric Structures, and their Applications\rq\rq \ (GNSAGA - INDAM). The second author  is partially supported by Prin 2015 (2015EYPTSB\_011 - {\em Geometry of Algebraic varieties}, Unit\`{a} locale Tor Vergata, CUP  E82F16003030006).

\bibliographystyle{amsplain}
\providecommand{\bysame}{\leavevmode\hbox to3em{\hrulefill}\thinspace}
\providecommand{\MR}{\relax\ifhmode\unskip\space\fi MR }
\providecommand{\MRhref}[2]{%
  \href{http://www.ams.org/mathscinet-getitem?mr=#1}{#2}
}
\providecommand{\href}[2]{#2}

\end{document}